\documentclass[reqno]{amsart}
\usepackage{amssymb,amsmath,amsthm}
\usepackage{thmtools}
\usepackage{graphicx} 
\usepackage{braket}
\usepackage{tikz,tikz-cd}
\usepackage{xparse}
\usepackage[maxbibnames=5,eprint=false,isbn=false,doi=false]{biblatex}
\usepackage{enumerate}
\usepackage[colorlinks=true,allcolors=blue,hypertexnames=false]{hyperref}
\usepackage[capitalise]{cleveref}

\newcommand\norm[1]{\left\lVert#1\right\rVert}

\newtheorem*{lemma*}{Lemma}
\newtheorem*{thm*}{Theorem}
\newtheorem*{cor*}{Corollary}
\newtheorem*{prop*}{Proposition}
\newtheorem*{claim*}{Claim}

\newtheorem{thm}{Theorem}[subsection]

\newtheorem{lemma}[thm]{Lemma}
\newtheorem{cor}[thm]{Corollary}
\newtheorem{prop}[thm]{Proposition}

\theoremstyle{definition}
\newtheorem*{notation*}{Notation}
\newtheorem*{defn*}{Definition}
\newtheorem*{case*}{Case}
\newtheorem*{subcase*}{Subcase}
\newtheorem*{step*}{Step}
\newtheorem*{rem*}{Remark}
\newtheorem*{const*}{Construction}
\newtheorem*{summary*}{Summary}
\newtheorem*{subclaim*}{Sub-claim}
\newtheorem*{ex*}{Example}
\newtheorem*{question*}{Question}

\newtheorem{defn}[thm]{Definition}

\newtheorem{rem}[thm]{Remark}
\newtheorem{const}[thm]{Construction}

\newtheorem*{assumption*}{Assumption}

\newcommand{\mycomment}[1]{}

\newcommand{\C}{\mathbb{C}}
\newcommand{\R}{\mathbb{R}}
\newcommand{\Z}{\mathbb{Z}}
\newcommand{\N}{\mathbb{N}}

\newcommand{\bK}{\mathcal{K}} 
\newcommand{\cB}{\mathcal{B}} 
\newcommand{\cC}{\mathcal{C}} 
\newcommand{\cG}{\mathcal{G}} 
\newcommand{\wcG}{\widetilde{\mathcal{G}}} 
\newcommand{\wcU}{\widetilde{U}}
\newcommand{\wcF}{\widetilde{F}}
\newcommand{\bI}{\mathbb{I}}  
\newcommand{\Top}{\mathcal{U}} 
\newcommand{\TCat}{\mathcal{U}\text{\normalfont Cat}} 
\newcommand{\TCatmon}{\mathcal{U}\text{\normalfont Cat}_{\ostar}} 
\newcommand{\Mod}{\mathcal{M}\text{\normalfont od}}
\newcommand{\Cuntz}[1]{\mathcal{O}_{#1}}
\newcommand{\CrossMod}[3]{
\begin{tikzcd}[ampersand replacement=\&,column sep=1.3cm]
    #2 \ar[r,"#3"] \& #1 \circlearrowleft #2
\end{tikzcd}
}
\newcommand{\ostar}{\circledast}

\newcommand{\mor}[1]{\text{\normalfont mor}(#1)}
\newcommand{\ob}[1]{\text{\normalfont ob}(#1)}
\newcommand{\op}{\text{\normalfont op}}
\newcommand{\Func}[2]{\text{\normalfont Func}(#1,#2)}

\DeclareMathOperator{\Aut}{Aut}
\DeclareMathOperator{\Inn}{Inn}
\DeclareMathOperator{\Out}{Out}
\DeclareMathOperator{\Ad}{Ad}
\DeclareMathOperator*{\hocolim}{hocolim}
\DeclareMathOperator{\id}{id}
\DeclareMathOperator{\spec}{spec}

\title[Infinite loop spaces and group actions on SSA $C^*$-algebras]{Infinite loop spaces and group actions \\ on strongly self-absorbing $C^*$-algebras}

\author{Ulrich Pennig}
\address{Ulrich Pennig, School of Mathematics, Cardiff University, Cardiff, CF24 4AG, Wales, UK}
\email{pennigu@cardiff.ac.uk}

\author{George Tridimas}
\address{George Tridimas, School of Mathematics, Cardiff University, Cardiff, CF24 4AG, Wales, UK and Heilbronn Institute for Mathematical Research, Bristol, UK}
\email{tridimasg@cardiff.ac.uk}

\thanks{This work was supported by the Additional Funding Programme for Mathematical Sciences, delivered by EPSRC (EP/V521917/1) and the Heilbronn Institute for Mathematical Research}

\begin{document}

\begin{abstract}
    Lifting obstructions for group actions, cocycle actions, and $\Gamma$-kernels admit a cohomological description via topological crossed modules, as recently developed by Izumi, Girón-Pacheco, and the first named author. For a strongly self-absorbing $C^*$-algebra $A$, we show that the classifying spaces $\cB^D\cG_A$ and $\cB^DP\cG_A$ of the respective crossed modules governing cocycle actions and $\Gamma$-kernels, respectively, carry infinite loop space structures induced by the tensor product; the same holds for the crossed module $B^D\wcG_A$ whenever $U(A)$ is connected. This confirms a conjecture from the aforementioned work and extends to $\Gamma$-kernels and cocycle actions the connection with stable homotopy theory. For the proof we construct $\bI$-FCPs from the relevant crossed modules, pass to $\Gamma$-spaces, and apply the May–Thomason infinite loop space machine. Consequently, the natural transformation $H^1(\Gamma,\cG) \to [\cB\Gamma,\cB^D\cG]$ takes values in cohomology groups.
\end{abstract}

\maketitle

\tableofcontents

\section*{Introduction}
A countable discrete group $\Gamma$ can act in various different ways on a unital $C^*$-algebra $A$: first, we have the natural notion of a group action given by a homomorphism $\beta \colon \Gamma \to \Aut(A)$. A $\Gamma$-kernel on the other hand is a homomorphism $\bar{\alpha} \colon \Gamma \to \Out(A)$, where $\Out(A) = \Aut(A)/\Inn(A)$ is the quotient by the group of unitarily implemented inner automorphisms. A cocycle action consists of two maps $\alpha \colon \Gamma \to \Aut(A)$ and $u \colon \Gamma \times \Gamma \to U(A)$, where $u$ measures the failure of $\alpha$ to be a homomorphism and satisfies an associativity constraint:
\begin{gather*}
    \alpha_g \circ \alpha_h = \Ad_{u_{g,h}}{} \circ \alpha_{gh}\ , \\
    \alpha_g(u_{h,k})u_{g,hk} = u_{g,h}u_{gh,k}\ .
\end{gather*}
Lifting a $\Gamma$-kernel to a cocycle action is not always possible and leads to an obstruction class in group cohomology living in $H^3(\Gamma, U(Z(A)))$. If $\Gamma$ is a countable discrete amenable group and $A = \mathcal{R}$ is the hyperfinite type $II_1$-factor, then this obstruction class is known to be a complete invariant for injective $\Gamma$-kernels up to conjugacy as proven by Connes \cite{Connes1975, Connes1977}, Jones \cite{Jones1978, Jones1980} and Ocneanu \cite{Ocneanu1985}. 

As the work of Evington and Gir\'on-Pacheco \cite{EvingtonPacheco2023} shows the situation for $C^*$-algebras is much more intricate even when we restrict to those algebras that could be considered $C^*$-algebraic analogues of the hyperfinite factors. For the Jiang-Su algebra~$\mathcal{Z}$, for example, the obstruction always vanishes \cite[Theorem~A]{EvingtonPacheco2023}. For UHF-algebras the order of the obstruction class is constrained by the supernatural number arising from the matrix dimensions \cite[Theorem~B]{EvingtonPacheco2023}. Nevertheless, classification results do exist in the setting of Kirchberg algebras. Let $\mathcal{OA}(\Gamma,A)$ be the set of outer actions of a torsion-free countable discrete amenable group $\Gamma$ on a stable Kirchberg algebra $A$. Combining the striking dynamical Kirchberg--Phillips theorem by Gabe and Szabo \cite[Theorem~6.2]{GabeSzabo2024} with Meyer's classification result \cite[Theorem~3.10]{Meyer2021} we obtain a bijection
\begin{equation} \label{eqn:OA-homotopyset}
    \mathcal{OA}(\Gamma, A)/\!\!\sim_{cc}\  \to\   [\cB\Gamma, \cB\!\Aut(A)] \ ,
\end{equation}
where the equivalence relation on the left hand side is cocycle conjugacy. The homotopy set that appears on the right involves the classifying spaces $\cB\Gamma$ and $\cB\!\Aut(A)$, where $\Aut(A)$ is considered as a topological group with the point-norm topology. This bijection was first conjectured in the work of Matui and Izumi on poly-$\Z$ group actions on Kirchberg algebras \cite{IzumiMatui2021-I,IzumiMatui2021-II}. 

Under mild assumptions on $X$ the homotopy set $[X,\cB\!\Aut(A)]$ is in bijection with locally trivial $C^*$-algebra bundles over $X$ with fibre $A$. So if $D$ is a strongly self-absorbing Kirchberg algebra and $A = D \otimes \bK$, then $(D \otimes \bK)^{\otimes 2} \cong D \otimes \bK$, which gives the right hand side of \eqref{eqn:OA-homotopyset} a semigroup structure arising from the fibrewise tensor product of $D \otimes \bK$-bundles. Using methods from stable homotopy theory, it was shown in \cite{Dadarlat_2015} that this is actually a group, in fact the first group of a generalised cohomology theory represented by a spectrum $gl_1(KU^D)$.

This raises the questions that motivated this paper, namely whether the connection to stable homotopy theory extends to $\Gamma$-kernels and cocycle actions in a natural way, and how the lifting obstructions manifest in this setting. For $\Gamma$-kernels on Kirchberg algebras an extensive analysis has been done in \cite{Izumi2024}. In \cite{pacheco2025gkernelscrossedmodules}, the first named author along with Izumi and Gir\'on-Pacheco developed a unified framework for various lifting obstructions for $\Gamma$-kernels in the general setting. Their picture is cohomological in nature. All known lifting obstructions for $\Gamma$-kernels and cocycle actions appear as boundary maps in exact sequences of cohomology sets. 

This is achieved by constructing a topological crossed module for each type of action. Since crossed modules can serve as coefficients for low-degree group cohomology (in a way that extends the classical definition) one obtains cohomology sets that, as shown in \cite{pacheco2025gkernelscrossedmodules}, are naturally isomorphic to the sets of respective $\Gamma$-actions on $A$ up to the right version of conjugacy. The table shown in \cref{fig:coh_w_coeff} summarises this relation.
\begin{figure}[htp]
    \centering
    \begin{tabular}{|c|c|c|}
         \hline
         \multicolumn{2}{|c|}{Crossed Module} & $H^1(\Gamma, \cG)$ \\[1mm]
         \hline
         & $1 \to \Aut(A)$ & $\{ \text{actions of $\Gamma$ on $A$} \}/\sim_c$ \\[1mm]  
         $\cG_A$ & $U(A) \to \Aut(A)$  & $\{\text{cocycle actions of  $\Gamma$ on $A$} \}/\sim_{\text{c.c}}$ \\[1mm]
         $P\cG_A$ & $PU(A) \to \Aut(A)$ & $\{\text{$\Gamma$-kernels on $A$} \}/\sim_{\text{c}}$  \\[1mm]
         $\wcG_A$ & $\wcU(A) \to \Aut(A)$ &  \\[1mm]
         \hline
    \end{tabular}
    \caption{Topological crossed modules and the interpretation of their associated cohomology sets.}
    \label{fig:coh_w_coeff}
\end{figure}

To extend the aforementioned connection with stable homotopy theory to $\Gamma$-kernels and cocycle actions we will employ tools from higher category theory. In particular, we will make use of the well-known correspondence between crossed modules and (strict) $2$-groups \cite[Sec.~3.3]{Noohi2007}, which already played a central role in~\cite{pacheco2025gkernelscrossedmodules}. Recall that a $2$-group (or categorical group) is a groupoid equipped with an additional functorial group structure, making it a group object in the category of groupoids. Since our starting point is topological crossed modules, the $2$-groups arising from them will carry a natural topology on their objects and morphisms; in other words, they are topological $2$-groups in the sense of \cite{BaezStevenson2009}.

Just as ordinary groups admit classifying spaces, so too do topological $2$-groups. Several natural models for these exist, two of which - denoted $\cB^D\cG$ and $\cB^\otimes\cG$ for a topological crossed module $\cG$ - will be recalled in \cref{def:classifying_space}. Since a cocycle representing a class in $H^1(\Gamma, \cG)$ can be viewed as a pseudofunctor, and classifying spaces are functorial with respect to these, we obtain a natural transformation 
\begin{equation} \label{eqn:nat_trafo}
    H^{1}(\Gamma, \cG) \to [\cB\Gamma, \cB^D\cG]\ .
\end{equation}
We refer to \cite[Lemma~4.4]{pacheco2025gkernelscrossedmodules} for details. Combining \cite[Theorem~4.15]{pacheco2025gkernelscrossedmodules} with one of the main results from \cite{Dadarlat_2015}, we find that 
\[
    [\cB\Gamma, \cB^D\cG_A] \cong [\cB\Gamma, \cB\!\Aut_0(A \otimes \bK)]
\]
carries a natural group structure whenever $A$ is a strongly self-absorbing $C^*$-algebra. In light of this, it was speculated in \cite{pacheco2025gkernelscrossedmodules} that, for every strongly self-absorbing $C^*$-algebra $A$, the classifying spaces $\cB^{D}\cG_A$, $\cB^DP\cG_A$ and $\cB^{D}\wcG_A$ all admit infinite loop space structures induced by the tensor product. The aim of this paper is to prove this conjecture. As a consequence, the natural transformation \eqref{eqn:nat_trafo} will take values in cohomology groups. In follow-up work we will explore how short exact sequences of crossed modules will give rise to long exact sequences interlinking the corresponding cohomology theories. Ultimately, this will lead to a topological understanding of the lifting obstructions for strongly self-absorbing $C^*$-algebras that will allow us to tackle some of the conjectures in~\cite{Izumi2024}.

We briefly recall the connection between spectra and cohomology theories. By definition an infinite loop space $K$ is weakly equivalent to the zeroth space in an $\Omega$-spectrum. This is a sequence $(K_n, \epsilon_n)_{n \in \N_0}$ of pointed spaces (in the category of $k$-spaces) together with weak equivalences
\[
    \epsilon_n \colon K_n \to \Omega K_{n+1}\ .
\]
Here, $\Omega Y$ denotes the space of based loops in $Y$ for a pointed space $Y$ equipped with the compact-open topology. In particular, $K \simeq K_0 \simeq \Omega^n K_n$, which means we may view $K_n$ as an $n$-fold delooping of $K_0$. By a classical result in algebraic topology, see for example \cite[Theorem~4.58]{hatcher2001algebraic} for a proof, the functor $X \mapsto [X, K_n]$ defines an (unreduced) cohomology theory on finite CW-complexes. Thus, having infinite loop space structures on $\cG_A$, $P\cG_A$ and $\wcG_A$ will show that \eqref{eqn:nat_trafo} is not only group-valued, but these groups are computable using tools like excision. 

Stable homotopy theory offers a rich collection of infinite loop space machines \cite{Adams1978} that produce spectra from combinatorial or category-theoretic data; one prominent example are $\Gamma$-spaces developed by Segal \cite{segal74}. Let $\Gamma^{\op}$ be the category of pointed finite sets and basepoint preserving maps. A $\Gamma$-space is a functor from $\Gamma^\op$ to spaces that satisfies conditions we recall in \cref{def:GammaSpace}. The simplest example of such a space arises from an abelian group. 

A crucial observation, going back to Segal's original work, is that, under a cofibrancy assumption, any $\Gamma$-space $X$ can be delooped to another $\Gamma$-space $\cB X$. Iterating this construction produces an $\Omega$-spectrum whose underlying infinite loop space is $X(1^+)$ (where $1^+ = \{0,1\}$ with basepoint $0$), provided $\pi_0(X(1^+))$ is an abelian group. May and Thomason later refined this framework by introducing a whiskering functor that takes care of the cofibrancy \cite{MAY1978205}. This is the machine that we will use.  

In our setting, all $\Gamma$-spaces arise from diagrams in topological spaces (more precisely $k$-spaces) defined as follows: Let $\bI$ be the category with objects $[n] = \{1,\dots,n\}$ for $n \in \N_0$ (where $[0] = \emptyset$) and injective maps as its morphisms. This category has a monoidal structure given by the disjoint union. An $\bI$-FCP (short for functor with cartesian product) is a commutative monoid in the category of $\bI$-shaped diagrams in spaces\footnote{These have appeared under many other names in the literature before, for example commutative $\bI$-monoids \cite{Dadarlat_2015}.}. They give rise to $\Gamma$-spaces using a construction that has appeared several times in the literature (see for example \cite[Section~12]{Lind_2013}). Hence, to get to a connective $\Omega$-spectrum we chain the following machines together: 
\[
    \bI\text{-FCPs} \quad \rightsquigarrow \quad \Gamma\text{-spaces} \quad \rightsquigarrow \quad \Omega\text{-spectra}\ .
\]
The underlying infinite loop space of the $\Omega$-spectrum resulting from an $\bI$-FCP $F$ is $X_F = \hocolim_\bI F$.

Since an $\bI$-FCP can be seen as an $\N_0$-graded monoid with extra structure it is well-suited to work with the minimal tensor product of $C^*$-algebras. It is straightforward to check that 
\[
    F_A([n]) = \cG_{A^{\otimes n}} \quad, \quad PF_A([n]) = P\cG_{A^{\otimes n}} \quad \text{and} \quad \wcF_A([n]) = \wcG_{A \otimes n}
\]
extend crossed module valued functors on $\bI$ together with a multiplicative structure. Applying a version of the classifying space functor object-wise turns these functors into the $\bI$-FCPs $\cB^2F_A$, $\cB^2PF_A$ and $\cB^2\wcF_A$ as we will see in \cref{monoids under day conv for our functors}. As outlined above, these give rise to $\Omega$-spectra with underlying infinite loop spaces
\[
    \hocolim_\bI \cB^2F_A \quad, \quad \hocolim_\bI \cB^2PF_A \quad \text{and} \quad \hocolim_\bI \cB^2\wcF_A \ ,
\]
respectively. At this point the assumption that $A$ is strongly self-absorbing becomes crucial. In combination with a version of B\"okstedt's lemma, which we show in \cref{boeksted lemma ref}, it allows us to identify the weak homotopy type to get 
\[
    \cB^D\cG_A \simeq \cB^2\cG_A \simeq \hocolim_\bI \cB^2F_A
\]
and likewise $\cB^DP\cG_A \simeq \hocolim_\bI \cB^2PF_A$ and $\cB^D\wcG_A \simeq \hocolim_\bI \cB^2\wcF_A$, which we will employ to prove our main result, \cref{thm:ssa_weak_htpy_type}.
\begin{thm*}
    Let $A$ be a SSA C*-algebra. Then, $\cB^{D}\cG_A$, $\cB^DP\cG_A$ admit infinite loop space structures. If $U(A)$ is connected, then so does $\cB^D\wcG_A$.
\end{thm*}

The paper is structured as follows: We collect the necessary preliminaries about $C^*$-algebras and homotopy theory in \cref{section 1}. The convenient model for our classifying spaces $\cB^2\cG$ uses the double nerve and is described in detail in \cref{new model subsection}. We show in \cref{weak equiv between pennig constr and B^2}, whose proof is in the appendix, that it is weakly equivalent to the ones in \cite{pacheco2025gkernelscrossedmodules}. In \cref{section 2} we collect some technical results about the point-set topology of $U(A)$, $PU(A)$ and $\wcU(A)$ that guarantee well-pointedness and that these topological groups and the classifying spaces of their associated crossed modules are indeed objects in the category of Hausdorff $k$-spaces. Their slightly technical proofs have been moved to the appendix. We then prove the version of B\"okstedt's lemma that is needed in \cref{section 3}. The main result and its proof is contained in \cref{section 4}, where we define the functors $F_A$, $PF_A$ and $\wcF_A$ and apply the infinite loop space machine. 


\section{Preliminaries}\label{section 1}
In this section we will introduce the necessary background about the class of $C^*$-algebras we will consider and the machinery from stable homotopy theory that we will need. 

\subsection*{Conventions.}
We will work in the following category of topological spaces and with the following notion of topological category. We refer to  \cite{Steenrod1967} for details.
    \begin{enumerate}[i)]
        \item Denote by $\Top$ the category of Hausdorff $k$-spaces and continuous maps. The term ``space'' will always refer to an object of $\Top$.
        \item Let $\TCat$ be the category of topological categories $\cC$ and continuous functors. Objects $\ob{\cC}$ and morphisms $\mor{\cC}$ of $\cC$ are objects in $\Top$ and the inclusion of the identities $\ob{\cC} \to \mor{\cC}$ and also the source and target maps $s,t \colon \mor{\cC} \to \ob{\cC}$ are continuous. Likewise, a continuous functor is a functor $F \colon \cC \to \mathcal{D}$ between topological categories, such that the maps on objects and morphisms are continuous. \\ The objects of $\TCat$ are further required to have the property that the inclusion of the identities $\ob{\cC} \to \mor{\cC}$ is a (closed) cofibration. This is a technical requirement that aids in the construction of the homotopy colimit of a functor (see \cref{double bar construction section}) and is related to \cref{pointwise weak equiv realize to weak equiv}.
    \end{enumerate}

\subsection{Strongly self-absorbing \texorpdfstring{$C^*$}{CStar}-algebras}
The class of strongly self-absorbing $C^*$-algebras was introduced by Toms and Winter \cite{TomsWinter2007}. They play a cornerstone role in the classification programme for separable, simple, nuclear $C^*$-algebras in the sense that some important properties of these algebras are equivalent to tensorial absorption of a strongly self-absorbing one. Purely infiniteness, for example, corresponds to tensorial absorption of $\Cuntz{\infty}$. For a unital $C^*$-algebra $A$ we denote its unitary group by $U(A)$. Throughout the paper we will denote by $\otimes = \otimes_{\text{min}}$ the minimal tensor product of $C^*$-algebras.

\begin{defn}\label{defn of ssa algebra}
    A separable unital $C^*$-algebra $A$ is called \emph{strongly self-absorbing} (abbreviated: SSA) if there is a *-isomorphism $\psi \colon A \to A \otimes A$ and a continuous map $u \colon [0, 1) \to U(A)$ such that 
    \[
        \lim_{t \to 1^{-}} \norm{\psi(d) - u_t(d \otimes 1)u_t^{\ast}} = 0 
    \]
    for all $d \in A$.
\end{defn}

Due to the recent progress in the classification programme  (see \cite[Corollary~D]{TikuisisWhiteWinter2017} for the concrete result) the following is a complete list of strongly self-absorbing $C^*$-algebras that satisfy the UCT. Here, $\mathcal{Z}$ denotes the Jiang-Su algebra, $M_P$ for a set of primes $P$ is the UHF-algebra constructed from matrix algebras $M_p(\C)$ for $p \in P$, $\mathcal{Q}$ is the universal UHF-algebra and $\Cuntz{n}$ for $n \in \{2,\infty\}$ are Cuntz algebras. An arrow in the diagram points to the algebra that tensorially absorbs the one at the root of the arrow. 
\[
\begin{tikzcd}[row sep=0.1cm]
{} &\mathcal{Z}\ar[dd] \ar[r]  & M_P \ar[dd]\ar[r] & {\mathcal{Q}}\ar[dd]\ar[dr] &{}\\
{\C}\ar[ur]\ar[dr] &{} & {}  & {}&{\Cuntz{2}}\\
{} &\Cuntz{\infty} \ar[r]  & \Cuntz{\infty} \otimes M_P\ar[r] & {\Cuntz{\infty} \otimes \mathcal{Q}}\ar[ur] &{}	
\end{tikzcd}
\]

\begin{defn} \label{def:PUA_wcUA}
    Let $A$ be a unital $C^*$-algebra. The \emph{projective unitary group} $PU(A)$ is the quotient of $U(A)$ by its centre $ZU(A)$. If $A$ is simple (which is the case if $A$ is SSA), then $ZU(A) = \mathbb{T} \cdot 1_A =: U(1)$.

    If in addition $A$ is such that $U(A)$ is connected, then we denote by $\wcU(A)$ the \emph{universal cover} of $U(A)$. We have $\wcU(A) = \mathcal{P}_{U(A)}/\Omega_0U(A)$, where 
    \[
        \mathcal{P}_{U(A)} = \{ \gamma \in C([0,1],U(A))\ :\ \gamma(0) = 1_A\} 
    \]
    and $\Omega_0U(A)$ is the normal subgroup of the loops in $U(A)$ based at $1_A$ that are based homotopic to the constant loop.
\end{defn}


\subsection{Simplicial spaces, geometric realisations and nerves of categories}\label{nerves of cats section}
Simplicial sets describe topological spaces that are glued together from simplices, where, for example, $0$-simplices are points, $1$-simplices are edges and $2$-simplices are triangles. They can be seen as the blueprints for a space, where the passage to the actual space is performed by a functor called geometric realisation. More generally, we can consider simplicial objects in an arbitrary category $\cC$. We recall their definition here.

\begin{defn}[$\Delta$ category]
    The \emph{simplex category} $\Delta$ has as its objects the sets $[n] = \{0, 1,...,n\}$ (for $n \geq 0$). Its morphisms are all order preserving functions $[n] \to [m]$. Let $\Delta_{\text{inj}} \subset \Delta$ be the category with the same objects, but with only the injective maps as morphisms. 
\end{defn}

\begin{defn}
   For any category $\cC$, (covariant) functors $\Delta^{\op}\to \cC$ are called \emph{simplicial objects} of $\cC$. For instance:
    \begin{itemize}
        \item A simplicial set is a functor $\Delta^{\op} \to \text{Sets}$
        \item A simplicial space is a functor $\Delta^{\op} \to \Top$
    \end{itemize}
    For a simplicial object $S$ we denote $S([n])$ by $S_n$.
\end{defn}

We recall two useful functors that turn simplicial sets or spaces into topological spaces. For each $n \in \N_0$ the topological simplex $\Delta^n$ is given by
\[
    \Delta^n = \{ (t_0, \dots, t_n) \in \R^{n+1} : 0 \leq t_i \leq 1 \text{ for } i \in \{0,\dots, n\} \text{ and }  t_0 + \dots + t_n = 1 \}\ .
\]
Mapping $[n]$ to $\Delta^n$ and $f \in \Delta([n], [m])$ to 
\[
    f_\ast \colon \Delta^{n} \to \Delta^{m} \quad, \quad f_*(t_0, \dots, t_n) = \left(\sum_{i \in f^{-1}(0)} t_i, \sum_{i \in f^{-1}(1)} t_i, \dots, \sum_{i \in f^{-1}(m)} t_i\right)
\]
gives rise to a functor $\Delta \to \Top$.

\begin{defn}[Geometric realisation]\label{defn:geom_realization}
Let $S$ be a simplicial space. The \emph{thin geometric realisation} of $S$ is defined by 
\[
    \lvert S \rvert = \coprod_{n \geq 0} S_n \times \Delta^{n}/\!\!\sim \ ,
\]
where the equivalence relation is generated by 
\[
    (S(f)(x), y) \sim (x, f_*(y))
\]
for all $f \in \mor{\Delta}$, $x \in S_{t(f)}$ and $y \in \Delta^{s(f)}$, where $t, s$ are the source and target maps. The \emph{fat geometric realisation} is denoted by $\norm{S}$ and is a quotient space of the same coproduct as above, but by the equivalence relation that is generated only by the morphisms $f$ in $\Delta_{\text{inj}}$. 
\end{defn}

By definition we have a natural map $\norm{S} \to \lvert S\rvert$. Under favourable circumstances, that we will discuss now, this map is a weak equivalence. Let $\sigma^{n, i} \colon [n+1] \to [n]$ be the unique order-preserving map that hits $i$ twice. These maps are called codegeneracy maps with $S(\sigma^{n,i})$ being the degeneracies.

\begin{defn}[Reedy cofibrancy]\label{reedy cofib}
Given a simplicial space $S$, we define the $(n+1)$st \emph{latching object} of $S$ by 
\[
L_{n+1}S := \bigcup_{0 \leq i \leq n}S(\sigma^{n,i})(S_{n}) \subseteq S_{n+1}
\] 
which is a subspace of $S_{n+1} = S([n+1])$. $S$ is called a \emph{Reedy cofibrant} simplicial space if the inclusion $L_{n+1}S \subseteq S_{n+1}$ is a closed cofibration for all $n \geq 0$.
\end{defn}

We recall the following relevant result. A proof can be found in \cite[Proposition~1]{tomDieck1973/74}.

\begin{prop}\label{fat thin weak equiv for reedy cofib}
    If $S$ is a Reedy cofibrant simplicial space such that for all $n \geq 0$ the space $S_n$ is a Hausdorff $k$-space, then the canonical map $\norm{S} \to |S|$ is a weak equivalence.
\end{prop}

The following result is crucial for the rest of the paper. A proof can be obtained by combining the previous proposition with \cite[Theorem~2.2]{Ebert_2019} (see also \cite[Chapter~VII, Proposition 3.6]{GoerssJardine2009}).

\begin{lemma}\label{pointwise weak equiv realize to weak equiv}
Let $X, Y$ be Reedy cofibrant simplicial Hausdorff $k$-spaces and let $\alpha \colon X \to Y$ be a natural transformation. If, for all $n \geq 0$, $\alpha_n$ is a weak equivalence, then $|\alpha| \colon |X| \to |Y|$ is a weak equivalence.
\end{lemma}

Finally, we recall the following property of the thin geometric realisation, which is proven in \cite[Theorem~11.6]{may1972geometry}.

\begin{prop}\label{B commutes with products}
Let $A, B$ be a simplicial spaces such that for each $n \geq 0$, $A_n, B_n$ are Hausdorff $k$-spaces. Then, the map induced by the projections $$|A \times B| \xrightarrow{\cong} |A| \times |B|$$ is a homeomorphism. The product on the right is the $k$-product (the product in $\Top$).
\end{prop}

Since each $[n]$ is a totally ordered poset, we may view it as a category, which we will denote in the same way. Any order-preserving map $f \colon [n] \to [m]$ gives rise to a functor $[n] \to [m]$ that we will also denote by $f$. Given any (small) category $\cC$, one can then construct a simplicial set $N\cC \colon \Delta^{\op} \to \text{Set}$. On objects it is defined  by 
\[
    N_n\cC := \Func{[n]}{\cC}\ .
\] 
Given $f \colon [n] \to [m]$ we have an induced map $f^* \colon N_m\cC \to N_n\cC$ given by $F \mapsto F \circ f$, and $N\cC(f) = f^*$ defines the functor $N\cC$ on morphisms.

\begin{rem*}
    Recall that $N_n\cC := \Func{[n]}{\cC}$ for any category $\cC$. Because $[n]$ is the totally ordered poset category with $n+1$ elements, $N_n\cC$ can be bijectively identified with just a chain of $n$ composable morphisms for $n \geq 1$. For $n = 0$, $N_0\cC := \Func{[0]}{\cC}$ is the set of functors from the category with one object and one identity morphism to $\cC$ and so, $N_0\cC$ is just the set $\ob{\cC}$. So, when $|\ob{\cC}| = 1$, $N_0\cC$ cotnains only a single element. In particular, this is the case when $\cC$ is the category associated to any group.
\end{rem*}

If $\cC$ is a topological category, then $N\cC$ becomes a simplicial space. Indeed, by the above remark $N_n\cC$ is a subspace of $(\mor{\cC})^{n}$ for $n \geq 1$ and $N_0\cC = \ob{\cC}$.

\begin{defn} \label{def:BC}
    Let $\cC$ be an object of $\TCat$. The space $\cB\cC := |N\cC|$ is called the \emph{classifying space} of $\cC$.
\end{defn}

Additional structure on $\cC$ will lead to similar ones on the classifying space. The most prominent example is the following:

\begin{lemma}\label{strict monoidal realizes to top monoid}
    Let $\cC$ be an object of $\TCat$ with a strict monoidal structure $(\cC, \ostar, 1)$ such that the bifunctor $\ostar$ is continuous. Then, $\cB\cC$ is a topological monoid.
\end{lemma}
\begin{proof}
    As $\cC$ is an object of $\TCat$, $\ob{\cC}$ and $\mor{\cC}$ are both objects of $\Top$ and the source and target maps are continuous. As such, for all $n \geq 0$, $N_n\cC$ is a subspace (by definition) of $(\mor{\cC})^n$ that is closed. Hence, $N_n\cC$ is a Hausdorff $k$-space. From the strict monoidal structure of $\cC$, one has the following commutative diagrams:
    \[
    \begin{tikzcd}[cramped]
	(\cC \times \cC) \times \cC & \cC \times (\cC \times \cC) & \cC \times \C \\
	\cC \times \cC && \cC
	\arrow[shift right, no head, from=1-1, to=1-2]
	\arrow[no head, from=1-1, to=1-2]
	\arrow["{\ostar \times 1}"', from=1-1, to=2-1]
	\arrow["{1 \times \ostar}"', from=1-2, to=1-3]
	\arrow["{\ostar }"', from=1-3, to=2-3]
	\arrow["{\ostar }", from=2-1, to=2-3]
    \end{tikzcd}
    \] 
    \[
    \begin{tikzcd}[cramped]
	1 \times \cC & \cC \times \cC & \cC \times 1 \\
	& \cC
	\arrow[from=1-1, to=1-2]
	\arrow[from=1-1, to=2-2]
	\arrow["\ostar"', from=1-2, to=2-2]
	\arrow[from=1-3, to=1-2]
	\arrow[from=1-3, to=2-2]
    \end{tikzcd}
    \] 
    where $1$ is category with one object and one morphism (the identity). Sequentially applying the functors $N$ and $|\cdot|$ and using \cref{B commutes with products} and that $\cB1 = \ast$, we obtain the commutative diagrams required to show that $\cB \cC$ is a monoid in $\Top$.
\end{proof}

\begin{defn} \label{def:top_mon_cat}
    Let $\TCatmon \subset \TCat$ be the subcategory of topological strict monoidal categories and monoidal functors.
\end{defn}

\subsection{Two-sided bar construction and homotopy colimits}\label{double bar construction section}

The two-sided bar construction has a plethora of good properties that make it suitable as a model of the homotopy colimit. We recall the ones we need here.

Throughout this section, all topological categories $\cC, \cC'$ are assumed to be small with discrete object space and have the property that the inclusion of the identities $x \mapsto 1_x$ is a closed cofibration.

\begin{defn}
We use the following naming convention: A \emph{right $\cC$-module} is a contravariant functor from $\cC$ to $\Top$. A \emph{left $\cC$-module} is a covariant functor from $\cC$ to $\Top$.
\end{defn}

Note that $\cC$ can be seen as a monoid in the category of $\ob{\cC}$-graphs in the sense of \cite[Section~12]{May1975}. Moreover, if $X$ is a right $\cC$-module, then 
\[
    \mathbb{X} = \coprod_{c \in \ob{\cC}} X(c)
\]
naturally has the structure of a right $\ob{\cC}$-graph over $\cC$. Similarly, we have a left $\ob{\cC}$-graph $\mathbb{Y}$ over $\cC$ associated to a left $\cC$-module $Y$. The following category will be the domain of the two-sided bar construction.

\begin{defn}[$\Mod$ category]
 The category $\Mod$ has objects all triples $(X, \cC, Y)$ where and $X$ is a right $\cC$-module and $Y$ is a left $\cC$-module. 
    A morphism $(X, \cC, Y) \to (X', \cC', Y')$ in $\Mod$ is a triple $(\alpha, F, \beta$) consisting of
    \begin{enumerate}[(i)]
        \item A functor $F \colon \cC \to \cC'$ where the object/morphism maps are continuous
        \item A natural transformation $\alpha \colon X \to X' \circ F^{\op}$
        \item A natural transformation $\beta \colon Y \to Y' \circ F$.
    \end{enumerate} The composition and the identities are defined in the obvious way.
\end{defn}

The point is that on this category we have a natural functor $\Mod \to s\Top$ where $s\Top$ is the category of simplicial Hausdorff $k$-spaces.

\begin{defn}[Two-sided bar construction] \label{defn:bar}
There is a functor $B_{\ast} \colon \Mod \to s\Top$ defined as follows: $B_{m}(X, \cC, Y)$ is the subspace of 
\[
    \mathbb{X} \times ( \mor{\cC})^{m} \times \mathbb{Y}
\] 
such that for any $(x, f_1, ..., f_m, y) \in B_m(X, \cC, Y)$ we have:
\[
    x \in X(t(f_1)) \quad, \quad s(f_i) = t(f_{i+1}) \quad, \quad y \in Y(s(f_m))\ .
\]
The $i$-th degeneracy is given by inserting the identity on the object $s(f_i) = t(f_{i+1})$. For $i = m$, the face map is obtained by erasing $f_m$ and replacing $y$ by $Y(f_m)(y)$. For $i = 0$, the face map is obtained by erasing $f_1$ and replacing $x$ by $X(f_1)(x)$. For $0 < i < m$, the face map is defined by replacing $f_{i-1}$ and $f_{i}$ by the composition $f_{i}\circ f_{i-1}$ in-between the remaining morphisms.
\end{defn}

The functor $B_\ast$ and the geometric realisation give rise to a convenient model for the homotopy colimit of a functor.

\begin{defn}[Homotopy colimit]\label{hocolim model}
Let $Z \colon \cC \to \Top$ be a functor. One has the constant right $\cC$-module at $\ob{\cC}$ denoted by $\ast_\cC$. We define
\[
    \hocolim_\cC Z := |B_\ast(\ast_\cC, \cC, Z)|
\]
\end{defn}

If $\mor{\cC}$ has the discrete topology, then the following two lemmas are often useful (we use them in the proof of B\"okstedt's Lemma, see \cref{boeksted lemma ref}).

\begin{lemma}\label{behavior of bar const with projections}
    Let $\cC$ be a small category and give $\mor{\cC}$ and $\ob{\cC}$ the discrete topologies. Then, for any functor $X \colon \cC \to \Top$, there are homeomorphisms
    \[
        \hocolim_{\cC \times \cC} (X \circ \pi_1) \cong (\hocolim_{\cC} X) \times \cB\cC
    \]
    and 
    \[
        \hocolim_{\cC \times \cC} (X \circ \pi_2) \cong (\hocolim_{\cC} X) \times \cB\cC\ ,
    \]
    where the maps $\hocolim_{\cC \times \cC} (X\circ \pi_i) \to \hocolim_{\cC} X$ are the maps induced by the projection maps $\pi_i$.
\end{lemma}
\begin{proof}
    Because the topologies on $\mor{\cC}$ and $\ob{\cC}$ are discrete, there is an isomorphism of simplicial spaces
    $B_\ast(\ast_{\cC \times \cC}, \cC \times \cC, X \circ \pi_1) \xrightarrow{\cong}B_\ast(\ast_{\cC}, \cC, X) \times N\cC$ and the simplicial map $B_\ast(\ast_{\cC \times \cC}, \cC \times \cC, X\circ \pi_1) \to B_\ast(\ast_{\cC}, \cC, X)$ is the one induced by the map in $\Mod$ induced by the functor $\pi_1 \colon  \cC \times \cC \to \cC$. Then, one uses \cref{B commutes with products}. The proof for $\pi_2$ is identical.
\end{proof}

\begin{lemma}\label{bar construction const functor}
    Let $\cC$ be a small category and give $\mor{\cC}$ and $\ob{\cC}$ the discrete topologies. If $X \in \ob{\Top}$ and $c_X \colon \cC \to \Top$ is the constant functor at $X$, then 
    \[
        \hocolim_{\cC} c_X \cong X \times \cB\cC\ .
    \]
\end{lemma}
\begin{proof}
    Because the topologies on $\mor{\cC}$ and $\ob{\cC}$ are discrete, there is an isomorphism of simplicial spaces
    $B_\ast(\ast_{\cC}, \cC, c_X) \xrightarrow{\cong}\Delta_X\times N\cC$ where $\Delta_X$ is the simplicial space constant at $X$. Then, one uses \cref{B commutes with products} and the homeomorphism $|\Delta_X| \cong X$.
\end{proof}

\mycomment{
One very important property of $\hocolim$ is the following cofinality criterion. A proof can be found in \cite[Section~10]{Dugger2008}.

\begin{lemma}[Cofinality]\label{cofinality criterion}
Let $\alpha \colon \cC \to \cC'$ be a functor such that for all $c' \in \ob{\cC'}$, the comma category $(c' \downarrow \alpha)$ is non-empty and has contractible classifying space. Then, for all functors $X \colon \cC' \to \Top$, the canonical map
$\hocolim_\cC (X \circ \alpha) \to \hocolim_{\cC'} X$ is a weak equivalence.
\end{lemma}
}

The final property of the homotopy colimit we need is the following. It follows from \cite[Proposition~4.7 and Remark~4.9]{Dugger2008}.

\begin{lemma}\label{natural weak equiv gives weak equiv}
    Let $F, G \colon \cC \to \Top$ be functors and $\alpha \colon F \to G$ a natural weak equivalence. Then, $\alpha$ induces a weak equivalence $\hocolim_{\cC} \alpha \colon \hocolim_{\cC} F \to \hocolim_{\cC} G$.
\end{lemma}

\subsection{Topological crossed modules and their classifying spaces}\label{new model subsection}
The treatment of lifting obstructions of group actions on C*-algebras in \cite{pacheco2025gkernelscrossedmodules} relies on the important concept of the topological crossed module.
\begin{defn}[Top.\ crossed module]
A \emph{topological crossed module} $(G, H, \partial, \cdot)$ consists of 
\begin{itemize}
    \item two topological groups $G, H$
    \item a continuous group homomorphism $\partial \colon H \to G$
    \item a continuous group action $\cdot \colon G \times H \to H$ of $G$ on $H$
\end{itemize} subject to the following two relations:
\begin{enumerate}[i)]
    \item $\partial(g \cdot h) = g \partial(h)g^{-1}$ for all $(g, h) \in G \times H$.
    \item $\partial(h_1) \cdot h_2 = h_1h_2h_1^{-1}$ for all $h_1, h_2 \in H$.
\end{enumerate}
We will either denote a crossed module by $(G,H,\partial,\cdot)$ or by 
\[
    \CrossMod{G}{H}{\partial}.
\]
\end{defn}

Let $A$ be a unital $C^*$-algebra. There are several topological crossed modules associated to it.

\begin{ex*}
    Endow $\Aut(A)$ with the point-norm topology and $U(A) \subset A$ with the subspace topology. Then, there is a topological crossed module given by
    \[
        \CrossMod{\Aut(A)}{U(A)}{\Ad}
    \]
    where $\Ad(u)(a) = uau^*$ for $u \in U(A)$ and the action is given by $\text{ev}$ with $\text{ev}(\alpha, u) = \alpha(u)$ for $\alpha \in \Aut(A)$ and $u \in U(A)$.
\end{ex*}

\begin{ex*}
    Let $\Aut(A)$ and $U(A)$ be the topological groups from the last example and let $PU(A) = U(A)/ZU(A)$, where $ZU(A)$ is the center of $U(A)$. Equipped with the quotient topology $PU(A)$ is a topological group. Let $p \colon U(A) \to PU(A)$ be the quotient homomorphism. We have the crossed module
    \[
        \CrossMod{\Aut(A)}{PU(A)}{\Ad \circ p},
    \]
    where the action is given by $\text{ev}'$ with $\text{ev}'(\alpha, [u]) = [\alpha(u)]$ for $\alpha \in \Aut(A)$ and $[u] \in PU(A)$.
\end{ex*}

\begin{ex*}
    Let $A$ be a unital $C^*$-algebra such that $U(A)$ is connected. Let $\Aut(A)$ be as in the previous examples and let $\wcU(A)$ be the universal cover of $U(A)$ with covering projection $\widetilde{p} \colon \wcU(A) \to U(A)$. We have the topological crossed module
    \[
        \CrossMod{\Aut(A)}{\wcU(A)}{\Ad \circ \widetilde{p}},
    \]
    where the action is given by $\widetilde{\text{ev}}$ with $\widetilde{\text{ev}}(\alpha, \widetilde{u}) = \alpha( \widetilde{u})$, where $\alpha \in \Aut(A)$. Here, we understand $\widetilde{u} \in \wcU(A)$ as in \cref{def:PUA_wcUA} and $\alpha$ is applied point-wise.
\end{ex*}

It is well-known (see for example \cite[Section~3.3]{Noohi2007}) that a crossed module naturally gives rise to a monoidal category in the following way.
 \begin{const}\label{def: topological strict monoidal cat associated to topological crossed module}
    Let $(G, H, \partial, \cdot)$ be a topological crossed module. Define the topological strict monoidal category $\cG$ as follows:
    \begin{itemize}
        \item $\ob{\cG} = G$ and $\mor{\cG} = H \times G$.
        \item The pair $(h,g_1) \in H \times G$ provides a morphism $g_1 \to g_2$, where $g_2 = \partial(h)g_1$. 
        \item The composition is given by $(h_2, \partial(h_1)g) \circ (h_1, g) = (h_2h_1, g)$.
        \item The strict monoidal structure is given by a functor $\ostar \colon \cG \times \cG \to \cG$ given by $(g_1, g_2) \mapsto g_1g_2$ on objects and by $((h_1, g_1), (h_2, g_2)) \mapsto (h_1(g_1 \cdot h_2), g_1g_2)$ on morphisms with unit the neutral element $e_G \in G$. 
    \end{itemize}
\end{const}

\begin{defn}\label{strict monoidal cats from crossed modules} 
For our main examples of topological crossed modules we introduce the following convenient notation for the associated monoidal categories:
    \begin{center}
    \begin{tabular}{ccr}
        $\CrossMod{\Aut(A)}{U(A)}{\Ad}$ & $\rightsquigarrow$ & $\cG_A$ \\
        $\CrossMod{\Aut(A)}{PU(A)}{\Ad \circ p}$ & $\rightsquigarrow$ & $P\cG_A$ \\
        $\CrossMod{\Aut(A)}{\wcU(A)}{\Ad \circ \widetilde{p}}$ & $\rightsquigarrow$ & $\wcG_A$
    \end{tabular}
    \end{center}
\end{defn}

There are several natural ways to associate simplicial spaces to topological $2$-categories and therefore several ways of defining classifying spaces. We recall two of them that play a crucial role in  \cite{pacheco2025gkernelscrossedmodules}.

\begin{defn}[Pennig/Izumi/Gir\'on-Pacheco] \label{def:classifying_space}
       Let $(G, H, \partial, \cdot)$ be a topological crossed module and denote by $\cG$ the topological strict monoidal category it induces.
    \begin{enumerate}[(1)]
        \item  We may view $\cG$ as a $2$-category with one object. 
        Let $N^{D}\cG$ be the Duskin nerve of $\cG$ described in \cite[Section~4]{pacheco2025gkernelscrossedmodules}. Define
        \[
            \cB^D\cG := |N^D\cG|\ .
        \]
        \item Let $N\cG$ be the nerve of $\cG$ when seen as a topological $1$-category. Since $G, H \in \ob{\Top}$, it follows that $|N\cG|$ is a topological group by \cref{crossed module cat gives top group after realization}. Define
        \[
            \cB^\otimes\cG := \norm{N|N\cG|}\ .
        \]
    \end{enumerate}
\end{defn}

It was shown in \cite[Theorem~4.9]{pacheco2025gkernelscrossedmodules} that $\cB^D\cG$ and $\cB^\otimes\cG$ are weakly equivalent if in the crossed module $(G,H,\partial,\cdot)$ the group $H$ is well-pointed. We now present a third model (under mild assumptions weakly equivalent to the other two) that is closely related to $\cB^\otimes\cG$. It has very appealing properties as a functor that the other two do not. 

\begin{defn} \label{defn:double_nerve}
        Let $(G, H, \partial, \cdot)$ be a topological crossed module with $G, H \in \ob{\Top}$. Let $\cG$ be the topological strict monoidal category it induces. Define 
        \[
            \cB^2\cG := |N|N\cG||\ .
        \]
\end{defn}

We observe the following:

\begin{lemma}\label{lem:nerve_of_crossed_mod_and_bar_construction}
       Let $(G, H, \partial, \cdot)$ be a topological crossed module with $G, H \in \ob{\Top}$. Let $\cG$ denote the category it induces. Let $\cC_H$ be the one-point topological category associated to the topological group $H$ and consider the functor $F \colon C_H^{\op} \to \Top$ which takes the object of $C_H^{\op}$ to $G$ and which takes the morphism $h^{\op}$ to the continuous map $G \to G$ with $g \mapsto \partial(h^{-1})g$. Let $\ast_H \colon C_H \to \Top$ be the constant functor at the one-point space. Then, there is an isomorphism of simplicial spaces:
       \[
            N_\ast\cG \xrightarrow{\cong} B_\ast(F, C_H, \ast_H)\ .
       \]
\end{lemma}
\begin{proof}
    By \cref{defn:bar}, we have equalities (products are taken in $\Top$) for all $n \geq 0$: 
    \[
        B_{n}(F, C_H, *) = G \times H^{n} \times \ast\ .
    \]
    For $n > 0$, we have a homeomorphism 
    \begin{align*}
        \mu_{n} \colon N_n \cG & \to B_n(F, C_H, \ast_H)\\
        ((h_1, g_1),(h_2, g_2),\dots,(h_n, g_n))&\mapsto ( \partial(h_1^{-1}h_2^{-1}\dots h_{n-1}^{-1})g_n, (h_1,h_2,\dots,h_n), \ast)\ ,
    \end{align*} 
    where $(h_i,g_i) \colon g_i \to g_{i+1}$. The element $h_1^{-1}h_2^{-1}...h_{n-1}^{-1}$ is taken to be the identity of~$H$ if $n = 1$.
     
    For $n = 0$, there is the homeomorphism $\mu_0 \colon N_0\cG \to B_0(F, C_H, \ast_H)$ given by $(1_H,g) \mapsto (g, \ast, \ast)$. One then easily checks that this family of maps indeed constitutes a natural transformation.
\end{proof}

We improve \cref{strict monoidal realizes to top monoid} in the following sense.

\begin{lemma}\label{crossed module cat gives top group after realization}
       Let $(G, H, \partial, \cdot)$ be a topological crossed module with $G, H \in \ob{\Top}$. Let $\cG$ denote the category it induces.  Then $\cB\cG$ is a topological group.
\end{lemma}
\begin{proof}
    From \cref{strict monoidal realizes to top monoid}, we know that $\cB\cG$ is a topological monoid. We need to define the group inversion. This can be done on the level of the nerve. Identifying each $n$-simplex of the nerve as a chain of $n$ composable morphisms in $\cG$, we can define a simplicial map $i \colon N\cG \to N\cG$ via:
    \begin{align*}
        i_n \colon N_n\cG &\to N_n\cG\\
        ((h_1, g_1), \dots, (h_n, g_n)) &\mapsto ((g_1^{-1} \cdot h_1^{-1}, g_1^{-1}), \dots, (g_n^{-1} \cdot h_n^{-1}, g_n^{-1}))
    \end{align*} The geometric realization of $i$ gives a continuous map $\cB\cG \to \cB\cG$ which is the inversion of the group. To see this, we use \cref{B commutes with products} and the fact that the following diagrams commute in the category of simplicial spaces 
    \[
    \begin{tikzcd}[cramped]
	{N\cG} && {N\cG \times N\cG} && {N\cG} \\
	&& {N\cG}
	\arrow["{i \times 1 \circ \Delta}", from=1-1, to=1-3]
	\arrow["c"', from=1-1, to=2-3]
	\arrow["{N(\ostar)}"{description}, from=1-3, to=2-3]
	\arrow["{1 \times i \circ \Delta}"', from=1-5, to=1-3]
	\arrow["c"', from=1-5, to=2-3]
    \end{tikzcd}
    \] 
    where $\Delta$ is the diagonal map $N\cG \to N\cG \times N\cG$ and $c$ is the map $N\cG \to N\cG$ such that $c_n \colon N_n\cG \to N_n\cG$ is constant at $((1_H, 1_G),...,(1_H, 1_G)) \in N_n\cG$.
\end{proof}

The main advantage of \cref{lem:nerve_of_crossed_mod_and_bar_construction} is that we can now impose cofibrancy conditions on $H$ and $G$, so that instead of the fat geometric realization we can use the thin one which has better categorical properties.

\begin{lemma}\label{weak equiv between pennig constr and B^2}
    Let $\cG = (G, H, \partial,\cdot)$ be a topological crossed module such that \begin{enumerate}
        \item $H, G \in \ob{\Top}$.
        \item The inclusions of the identities $1_H \hookrightarrow H$ and $1_G \hookrightarrow G$ are cofibrations.
    \end{enumerate}
    Then,
    \begin{enumerate}
        \item $\cB\cG$ is a well-pointed Hausdorff $k$-group.
        \item $\cB^2\cG$ is a Hausdorff $k$-space.
        \item The canonical map 
        \[
            B^{\otimes}\cG \to \cB^2\cG
        \]
        is a weak equivalence.
    \end{enumerate}
\end{lemma}
\begin{proof}
    In the appendix.
\end{proof}

\subsection{\texorpdfstring{$\bI$}{I}-FCPs, \texorpdfstring{$\Gamma$}{Gamma}-spaces and an infinite loop space machine} As mentioned in the introduction, the main goal of this paper is to show that ${\cB^D\cG_A, \cB^DP\cG_A}$ and ${\cB^D\wcG}_A$ all admit infinite loop space structures. 

To achieve this we first construct $\bI$-FCPs (see \cref{equiv characterization of commutative monoid}) and then use these to get $\Gamma$-spaces using \cite[Construction~12.1]{Lind_2013}. As described for example in \cite[Definition~3.6]{MAY1978205} (based on earlier results in \cite{segal74}) there is then an infinite loop space machine that takes $\Gamma$-spaces to weak $\Omega$-spectra. This passage is arranged so that any one of our three spaces of interest, i.e.\ ${\cB^D\cG_A, \cB^DP\cG_A}$ or ${\cB^D\wcG}_A$, is weakly equivalent to the $0$-th space of the resulting weak $\Omega$-spectrum,  which implies that they naturally admit infinite loop space structures.

\begin{defn}[The category $\bI$]
    Let $[n] := \{1,2,3,...,n\}$ ($n \in \mathbb{N}$)  and $[0] := \emptyset$. 
    The category $\bI$ is determined by:
    \begin{itemize}
        \item $\ob{\bI} = \{[n] : n \in \N_0\}$.
        \item The set of morphisms $\mor{\bI}$ consists of all injective functions $f \colon [m] \to [n]$ for $m, n \in \N_0$. Their composition is just the ordinary one and the identities are the identity functions $1_{[n]}$.
    \end{itemize}
\end{defn}
Moreover, $\bI$ is a symmetric monoidal with respect to the operation $\oplus \colon \bI \times \bI \to \bI$ given by $[m] \oplus [n] := [m + n]$ on objects. The morphisms $f \colon [m] \to [n]$ and $g \colon [k] \to [l]$ map to $f \oplus g \colon [m + k] \to [n + l]$ given by $f$ on the first $m$ elements of $[m + k]$ and by $g + l$ on the last $k$ elements of $[m + k]$. 

\begin{defn}[$\bI$-FCP]\label{equiv characterization of commutative monoid}
    A functor $F \colon \bI \to \Top$ is called an $\bI$-FCP  if and only if there is a natural transformation $\mu_{[m], [n]} \colon F([m]) \times F([n]) \to F([m] \oplus [n])$ ($m, n \geq 0$) and a morphism $\eta \colon \ast \to F([0])$ ($\ast$ is the one-point space) such that the following diagrams (with obvious unmarked arrows) commute: 
    \[
    \begin{tikzcd}[cramped]
	   {(F([m]) \times F([n])) \times F([k])} & {F([m]) \times  (F([n]) \times F([k]))} \\
	   {F([m] \oplus [n]) \times F([k])} & {F([m]) \times F([n] \oplus [k])} \\
	   {F(([m] \oplus [n]) \oplus [k])} & {F([m] \oplus ([n] \oplus [k]))}
	   \arrow[from=1-1, to=1-2]
	   \arrow[from=1-1, to=2-1]
	   \arrow[from=1-2, to=2-2]
	   \arrow[from=2-1, to=3-1]
	   \arrow[from=2-2, to=3-2]
	   \arrow[from=3-1, to=3-2]
    \end{tikzcd}
    \] 
    and 
    \[
    \begin{tikzcd}[cramped]
	   {\ast \times F([m])} & {F([0]) \times F([m])} && {F([m]) \times \ast} & {F([m]) \times F([0])} \\
	   {F([m])} & {F([0] \oplus [m])} && {F([m])} & {F([m] \oplus [0])}
	   \arrow["{\eta \times 1}", from=1-1, to=1-2]
	   \arrow[from=1-1, to=2-1]
	   \arrow["{\mu_{[0], [m]}}", from=1-2, to=2-2]
	   \arrow["{1 \times \eta}"', from=1-4, to=1-5]
	   \arrow[from=1-4, to=2-4]
	   \arrow["{\mu_{[m], [0]}}"', from=1-5, to=2-5]
	   \arrow[from=2-2, to=2-1]
	   \arrow[from=2-5, to=2-4]
    \end{tikzcd}
    \] 
    and
    \[
    \begin{tikzcd}[cramped]
	   {F([m]) \times F([n])} & {F([n]) \times F([m])} \\
	   {F([m] \oplus [n])} & {F([n] \oplus [m])}
	   \arrow[from=1-1, to=1-2]
	   \arrow["{\mu_{[m], [n]}}"', from=1-1, to=2-1]
	   \arrow["{\mu_{[n], [m]}}", from=1-2, to=2-2]
	   \arrow[from=2-1, to=2-2]
    \end{tikzcd}
    \]
\end{defn}

As a consequence of \cite[Proposition~22.1]{MMSS01}, an $\bI$-FCP $F$ is the same as a commutative monoid in $\Top^{\bI}$ with respect to the Day convolution.

From an $\bI$-FCP, it is possible to construct a $\Gamma$-space. These originated in the work of Segal (see \cite{segal74}). We give the (equivalent) formulation used in \cite{Lind_2013}.

\begin{defn}[$\Gamma$-space] 
\label{def:GammaSpace}
    Let $\Gamma^{\op}$ be the category of finite pointed sets $n^{+} := \{0, 1,2,...,n\}$ for $n \geq 1$ and $0^{+} := \{0\}$. A morphism $n^{+} \to m^{+}$ is a basepoint-preserving function $f \colon n^{+} \to m^{+}$, i.e.\ $f$ satisfies $f(0) = 0$. 
    A \emph{$\Gamma$-space} is a (covariant) functor $A \colon \Gamma^{\op} \to \Top$ such that:
    \begin{itemize}
        \item $A(0^{+})$ is weakly contractible.
        \item For all $n$, the map $A(n^+) \to A(1^+)^{n}$, induced by the maps $i_{k} \colon n^+ \to 1^+$ where $i_k(k) = 1$ and $i_k \equiv 0$ otherwise, is a weak equivalence.
    \end{itemize}
\end{defn}

Let $A$ be a discrete abelian group and define $A(n^+) = A^n$ with $A(0^+) = 1$. For a morphism $f \colon n^+ \to m^+$ let $A(f) \colon A^n \to A^m$ be given by
\[
    A(f)(a_1, \dots, a_n) = \left( \prod_{j \in f^{-1}(1)} a_j,\prod_{j \in f^{-1}(2)} a_j, \dots, \prod_{j \in f^{-1}(m)} a_j \right)\ .
\]
This is a (discrete) $\Gamma$-space (or a $\Gamma$-set). With this example in mind we can view general $\Gamma$-spaces as homotopical generalisations of abelian groups. The following proposition goes back to Schlichtkrull \cite[Section~5.2]{Schlichtkrull2004}. We will not need the full details and refer the reader to the exposition in \cite[Construction~12.1]{Lind_2013}.

\begin{prop}[$\Gamma$-space machine]\label{Linds gamma space machine}
    There is a functor from the category of $\bI$-FCPs to the category of $\Gamma$-spaces. The functor associates to each $\bI$-FCP $F$ a $\Gamma$-space $X_F$ with 
    \[
        X_F(0^+) = \ast \qquad \text{and} \qquad X_F(1^+) = \hocolim_\bI F\ .
    \]
\end{prop}

The idea now is to apply an appropriate infinite loop space machine to a $\Gamma$-space and hence obtain a weak $\Omega$-spectrum. First, we recall their definition.

\begin{defn}
    A \emph{weak $\Omega$-spectrum} is a sequence of based spaces $E_{n}$ ($n \geq 0$) and a sequence of based weak equivalences $E_{n} \xrightarrow{\simeq} \Omega E_{n+1}$. It is called \emph{connective} if each $E_{n}$ is $(n-1)$-connected. An \emph{infinite loop space} is any based space which is weakly equivalent to the $0$-th space of a weak $\Omega$-spectrum.
\end{defn}

Since we can shift an $\Omega$-spectrum $(E_n)_{n \in \N_0}$, it turns out that all of its spaces $E_n$ are infinite loop spaces. The weak $\Omega$-spectra constitute a category. A morphism of $\Omega$-spectra $(E_n) \to (E'_n)$ is a sequence of based maps $E_n \to E'_n$ that make the following diagrams commute:
\[
    \begin{tikzcd}[cramped]
	{E_n} & {E'_n} \\
	{\Omega E_{n+1}} & {\Omega E'_{n+1}}
	\arrow[from=1-1, to=1-2]
	\arrow[from=1-1, to=2-1]
	\arrow[from=1-2, to=2-2]
	\arrow[from=2-1, to=2-2]
    \end{tikzcd}
\]

For the passage from $\Gamma$-spaces to weak $\Omega$-spectra we follow \cite{MAY1978205}. Suppose that $A$ is a $\Gamma$-space with $A(0^+) = \ast$. In this case all spaces $A(n^+)$ have a basepoint given by the image of  $\ast = A(0^+) \to A(n^+)$ induced by the unique map $0^+ \to n^+$. As pointed out by Segal in \cite{segal74} we have a natural functor $\iota \colon \Delta^{\op} \to \Gamma^{\op}$ mapping $[n]$ to $n^+$. A morphism $f \colon [m] \to [n]$ is mapped to $f^* \colon n^+ \to m^+$, which sends all $i \in n^+$ with $f(j-1) < i \leq f(j)$ to $j$ and the rest to $0$. By abuse of notation we denote the geometric realisation $\lvert A \circ \iota \rvert$ by just $\lvert A \rvert$. In \cite{MAY1978205} the authors consider $\Gamma$-spaces that satisfy the additional cofibrancy condition \cite[Definition 1.2 (3)]{MAY1978205}. Let $\Gamma\text{-spaces}^{\text{cof}}$ be the corresponding subcategory of $\Gamma$-spaces. In \cite[Definition 3.6]{MAY1978205} they then define a functor 
\[
    E \colon \Gamma\text{-spaces}^{\text{cof}} \to \text{weak }\Omega\text{-Spectra}
\]
and a natural group completion map $\kappa \colon A(1^+) \to E(A)_0$. Given a $\Gamma$-space $A$ with $A(0^+) = \ast$ we define a new $\Gamma$-space by
\[
    \cB A(n^+) = \lvert m^+ \mapsto A(n^+ \wedge m^+) \rvert
\]
and refer the reader to \cite[Construction 3.4]{MAY1978205} for the definition of the smash product of finite pointed sets. In case $A \in \ob{\Gamma\text{-spaces}^{\text{cof}}}$ the spaces in the $\Omega$-spectrum $E(A)$ are obtained by iterating this procedure and evaluating at $1^+$. From \cite[Definition~1.5 and  Proposition~1.6]{MAY1978205}, we also have a functor 
\[
    W \colon \Gamma\text{-spaces} \to \Gamma\text{-spaces}^{\text{cof}}
\]
from our notion of $\Gamma$-spaces to theirs, which is a pointwise weak equivalence.

Let $F$ be an $\bI$-FCP. Then \cref{Linds gamma space machine} gives a $\Gamma$-space $X_F$ with $X_F(0^+) = \ast$ and $X_F(1^+) = \hocolim_{\bI}F$. We refer to \cref{double bar construction section} for the definition of $\hocolim_{\bI}$. The $\Gamma$-space structure gives rise to maps
\[
    (\hocolim_{\bI}F)^2 = X_F(1^+)^2 \xleftarrow{\simeq} X_F(2^+) \to X_F(1^+) = \hocolim_{\bI}F
\]
which give $\pi_0(\hocolim_{\bI}F)$ the structure of an abelian monoid. Putting everything together we now have:

\begin{prop}\label{inf loop space machine used}
    If the monoid $\pi_0(\hocolim_{\bI}F)$ is a group, then there is a connective weak $\Omega$-spectrum $(E_n)_{n \in \N_0}$ defined by $E = E(W X_F)$ together with a weak equivalence $\hocolim_{\bI}F \simeq E_0$. In particular, $E_0$ is an infinite loop space.
\end{prop}
\begin{proof}
    Since $X_F(0^+) = \ast$, we can view $X_F$ as a functor to based topological spaces. By abuse of notation we will continue to denote this functor by $X_F$. 
    From the above, we have that $WX_F(1^+) \xrightarrow{\simeq } X_F(1^+) = \hocolim_{\bI}F$ is a weak equivalence. By the functoriality of $W$ and the above, the induced map $\pi_0(WX_F(1^+)) \to \pi_0(X_F(1^+))$ is a monoid isomorphism. By construction, $\pi_0(X_F(1^+)) = \pi_0(\hocolim_{\bI}F)$ as monoids. Hence, $\pi_0(WX_F(1^+))$ is a group isomorphic to $\pi_0(\hocolim_{\bI}F)$. It follows by \cite[Lemma~2.2]{MAY1978205} that the natural group completion $WX_F(1^+) \to E_0(WX_F)$ is a weak equivalence.
\end{proof}


\section{Topological properties of unitary groups of C*-algebras} \label{section 2}
In this section we collect results about the topological structure of $\Aut(A)$, $U(A)$, $PU(A)$ and $\wcU(A)$ for a strongly self-absorbing $C^*$-algebra $A$. In particular, we will discuss their well-pointedness. As a reminder, a topological group is called well-pointed whenever the inclusion of its neutral element is a cofibration.




\begin{lemma}\label{big topological lemma for ua and pua}
    Let $A$ be a separable unital SSA C*-algebra. Then:
    \begin{enumerate}[(1)]
        \item $U(A)$ is well-pointed.
        \item $PU(A)$ is well-pointed.
        \item The quotient map $U(A) \to PU(A)$ is a principal $S^1$-bundle.
        \item The map $U(A) \to U(A \otimes A) : u \mapsto u \otimes 1$ is a weak equivalence.
        \item The map $PU(A) \to PU(A \otimes A) : [u] \mapsto [u \otimes 1]$ is a weak equivalence.
    \end{enumerate}
\end{lemma}
\begin{proof}
    In the appendix.
\end{proof}

As a corollary of the first point and of the group structure on $U(A)$, $U(A)$ is always locally contractible. Thus, if $U(A)$ is connected, a universal cover of $U(A)$ exists. Next we show that all of the spaces involved in our crossed modules are Hausdorff $k$-spaces.

\begin{lemma}\label{main groups are all hausdorff k-groups}
    Let $A$ be a separable unital SSA C*-algebra. Then, $U(A), PU(A)$ and $\Aut(A)$ are Hausdorff $k$-spaces. If $U(A)$ is connected, then the universal cover $\wcU(A)$ is also a Hausdorff $k$-space.
\end{lemma}
\begin{proof}
$U(A)$ has the subspace topology of $A$ and is, therefore, metrizable. All metrizable spaces are Hausdorff $k$-spaces \cite[Proposition~2.2]{Steenrod1967}. $\Aut(A)$ is given the point norm topology. As $A$ is unital and separable, $\Aut(A)$ is also metrizable and hence a Hausdorff $k$-space for the same reason as before. We have the quotient map $q \colon U(A) \to PU(A)$. 
$PU(A)$ is Hausdorff because the equivalence relation generated by $q$ is the subset $\{(u, v) \in U(A) \times U(A) : v^{-1}u \in U(1)\}$ which is clearly closed in $U(A) \times U(A)$ because $U(1)$ is closed in $U(A)$. This implies that the diagonal of $PU(A)$ is closed in $PU(A) \times PU(A)$ because the product $q \times q \colon U(A) \times U(A) \to PU(A) \times PU(A)$ is a quotient map as $q$ is open and surjective.
It follows by \cite[Proposition~2.6]{Steenrod1967} that $PU(A)$ is a Hausdorff $k$-space.  

If $U(A)$ is connected, then it has a universal cover as it is locally contractible. We have the covering $p \colon \wcU(A) \to U(A)$. As $U(A)$ is first countable because it is metrizable, it follows from the fact that $p$ is a local homeomorphism that $\wcU(A)$ is first countable. It is clear that $\wcU(A)$ is Hausdorff. By \cite[Proposition~2.2]{Steenrod1967}, $\wcU(A)$ is a Hausdorff $k$-space.
\end{proof}

The following can also easily be checked using \cref{big topological lemma for ua and pua} and the techniques from the proof of \cref{weak equiv between pennig constr and B^2} in the appendix.

\begin{lemma}
    Let $A$ be a separable unital C*-algebra. Then, $\cB^2U(A)$ and $\cB^2PU(A)$ are Hausdorff $k$-spaces. If $U(A)$ is connected, then $\cB^2\wcU(A)$ is a Hausdorff $k$-space.
\end{lemma}

Recall our preferred model for $\wcU(A)$ from \cref{def:PUA_wcUA}. We make the following remarks about it in the general setting.

\begin{rem}\label{universal-cover of group}
Whenever $G$ is a path-connected locally path-connected and semi-locally simply-connected topological group, there is a very concrete description of~$\widetilde{G}$. 

Let $\mathcal{P}G := \{\gamma \in C(I ,G) : \gamma(0) = e\}$ be the group of continuous paths in $G$ starting at $e$ (with pointwise multiplication of paths) and let $\Omega_0G$ be the subset of $PG$ with $\gamma(1) = e$ and $\gamma$ is homotopic relative to $\partial I$ to the constant loop at $e$. Then, $\Omega_0G$ is  normal subgroup of $\mathcal{P}G$. There is an evident bijection $$\mathcal{P}G/\Omega_0G \cong \widetilde{G}$$ where $\widetilde{G}$ is as it is constructed in \cite{hatcher2001algebraic}. We can define a topology on $\mathcal{P}G/\Omega_0G$ by taking this bijection to be a homeomorphism. The multiplication and inversion maps of $\mathcal{P}G/\Omega_0G$ turn out to be continuous. The covering map $p_G \colon \mathcal{P}G/\Omega_0G \to G$ is given by $[\gamma] \mapsto \gamma(1)$.
\end{rem}

One of the benefits of this construction is its naturality.

\begin{rem}\label{functorial construction of universal cover}
The above construction $G \mapsto \widetilde{G}$ is functorial in the category of path-connected, locally path-connected and semi-locally simply connected topological groups and continuous group morphisms.

The covering maps $p_G \colon \widetilde{G} \to G$ constitute a natural transformation from the functor $G \mapsto \widetilde{G}$ to the identity functor $G \mapsto G$.
\end{rem}

\section{A proof of B\"okstedt's lemma}\label{section 3}

In this section we will prove a version of B\"okstedt's lemma (see \cite[Lemma~2.1]{Schlichtkrull2004}), which is used to identify $\cB^2\mathcal{G_A}, \cB^2P\cG_A, \cB^2\wcG_A$ with infinite loop spaces. For each $n \geq 0$, let $\bI_{ \geq n}$ be the full subcategory of $\bI$ which contains only objects $m \geq n$ and similarly for $\bI_{> n}$.
We will need the following, which is proven in \cite[Lemma~6.4]{Mandell_2002}.

\begin{lemma}\label{hocolim inclusion equiv}
    Let $F \colon \bI \to \Top$ be a functor. For each $n \geq 0$, the subcategory inclusion $\bI_{\geq n} \to \mathbb{\bI}$ induces a homotopy equivalence:
    \[
        \hocolim_{\bI_{\geq n}}F \xrightarrow{\simeq } \hocolim_{\bI}F\ .
    \]
\end{lemma}

Using this, we can now prove B\"okstedt's lemma. The proof mimicks the proof of \cite[Lemma~2.5.1]{BRUN200029}.

\begin{thm}[B\"okstedt's lemma]\label{boeksted lemma ref}
    Assume that $X \colon \bI \to \Top$ is a functor such that there exists some integer  $n \geq 0$ with the property that $X(c) \to X(c')$ is a weak equivalence for all morphisms $c \to c'$ with $c, c' \in \bI_{>n}$. For any object $d \in \bI_{>n}$, there is a weak equivalence $$X(d)\xrightarrow{\simeq} \hocolim_{\bI}X$$
\end{thm}
    \begin{proof}
    By \cref{hocolim inclusion equiv}, it suffices to exhibit a weak equivalence 
    \[
        X(d) \xrightarrow{\simeq } \hocolim_{\bI_{>n}} X\ .
    \]
    Let us write $\mu \colon \bI \times \bI \to \bI$ for the monoidal structure in $\bI$ and $\pi_i \colon \bI \times \bI \to \bI$ $(i = 1,2)$ for the two projections. As $[0]$ is initial in $\bI$ and a unit for $\mu$, there are natural transformations $\alpha \colon \pi_1 \to \mu$ and $\beta \colon \pi_2 \to \mu$. For instance, $\pi_1([\lambda], [k]) \to \mu([\lambda], [k])$ is obtained by the composition $\pi_1([\lambda], [k]) = [\lambda] \xrightarrow{\cong}\mu([\lambda], [0]) \to \mu([\lambda], [k])$ where we have used the identity on $[\lambda]$ and the unique morphism $[0] \to [k]$ to get the rightmost arrow.
    
    Let us use the following notation for objects in $\Mod$ (see \cref{double bar construction section}) to save space: $(\ast_\cC, \cC, F)$ is denoted by $(\cC, F)$. Let us also denote by $\{d\} \times \bI_{>n}$ the subcategory of $\bI_{>n} \times \bI_{>n}$ which contains morphisms only of the form $1_d \times g$ for $g \in \mor{\bI_{>n}}$. 

Using the maps $\alpha, \beta$ and the subcategory inclusion $\{d\} \times \bI_{>n} \to \bI_{>n} \times \bI_{>n}$, we get a commutative diagram in $\Mod$:
\[
\begin{tikzcd}[cramped]
	& {(\bI_{>n}, X)} \\
	{(\{d\} \times \bI_{>n}, X\pi_1)} & {(\bI_{>n} \times \bI_{>n}, X\pi_1)} \\
	{(\{d\} \times \bI_{>n}, X\mu)} & {(\bI_{>n} \times \bI_{>n}, X\mu)} \\
	{(\{d\} \times \bI_{>n}, X\pi_2)} & {(\bI_{>n} \times \bI_{>n}, X\pi_2)} \\
	& {(\bI_{>n}, X)}
	\arrow["{\iota_\ast}", from=2-1, to=2-2]
	\arrow["{(X\alpha)_\ast}"', from=2-1, to=3-1]
	\arrow["{(\pi_1)_\ast}"', from=2-2, to=1-2]
	\arrow["{(X\alpha)_\ast}", from=2-2, to=3-2]
	\arrow["{\iota_\ast}", from=3-1, to=3-2]
	\arrow["{(X\beta)_\ast}", from=4-1, to=3-1]
	\arrow["{\iota_\ast}", from=4-1, to=4-2]
	\arrow["{{(\pi_2)}_\ast}"{description}, from=4-1, to=5-2]
	\arrow["\cong"', shift right=2, draw=none, from=4-1, to=5-2]
	\arrow["{(X\beta)_\ast}"', from=4-2, to=3-2]
	\arrow["{{(\pi_2)}_\ast}", from=4-2, to=5-2]
\end{tikzcd}
\] 
where by $x_\ast$ we mean the arrow in $\Mod$ induced by $x$. Notice that the $(\pi_2)_\ast$ arrow decorated by $\cong$ is an isomorphism in $\Mod$.

Taking homotopy colimits, we get a commutative diagram in $\Top$:
\[
\begin{tikzcd}[column sep=1.5cm,row sep=.5cm]
	& \displaystyle{\hocolim_{\bI_{>n}} X} \\
	\displaystyle{\hocolim_{\{d\} \times \bI_{>n}} X\pi_1} & 
    \displaystyle{\hocolim_{\bI_{>n} \times \bI_{>n}} X\pi_1} \\
	\displaystyle{\hocolim_{{\{d\} \times\bI_{>n}}} X\mu} & 
    \displaystyle{\hocolim_{\bI_{>n} \times \bI_{>n}} X\mu} \\
	\displaystyle{\hocolim_{\{d\} \times\bI_{>n}} X\pi_2} & 
    \displaystyle{\hocolim_{\bI_{>n} \times \bI_{>n}} X\pi_2} \\
	& \displaystyle{\hocolim_{\bI_{>n}} X}
	\arrow[from=2-1, to=2-2]
	\arrow["\simeq"', from=2-1, to=3-1]
	\arrow[from=2-2, to=1-2]
	\arrow["\simeq", from=2-2, to=3-2]
	\arrow[from=3-1, to=3-2]
	\arrow["\simeq", from=4-1, to=3-1]
	\arrow[from=4-1, to=4-2]
	\arrow["\cong"', from=4-1, to=5-2]
	\arrow["\simeq"', from=4-2, to=3-2]
	\arrow[from=4-2, to=5-2]
\end{tikzcd}
\]

The arrow labeled by $\cong$ is a homeomorphism by the functoriality of the construction of the homotopy colimit.
The vertical maps labeled by $\simeq$ are induced by the natural transformations $X\alpha, X\beta$ which are point-wise weak equivalences by our assumptions on $X$. Thus, their induced maps on homotopy colimits are weak equivalences by \cref{natural weak equiv gives weak equiv}. 

Taking the homotopy colimit of the constant functor $\bI \to \Top$ at the one-point space, it follows from \cref{hocolim inclusion equiv} and \cref{bar construction const functor} that there is a homotopy equivalence $$\cB\bI_{>n} \xrightarrow{\simeq} \cB\bI$$ Moreover, as $\bI$ has an initial element, it follows that $\cB\bI$ is contractible. Using the above and \cref{behavior of bar const with projections}, it follows that the arrows $$\hocolim_{\bI_{>n} \times \bI_{>n}}X\pi_1 \to \hocolim_{\bI_{>n}}X$$ and $$\hocolim_{\bI_{>n} \times \bI_{>n}}X\pi_2 \to \hocolim_{\bI_{>n}}X$$ are homotopy equivalences.

From the 2-out-of-3 rule for weak equivalences and the commutativity of the above diagram, it follows that the arrow $$\hocolim_{\{d\} \times \bI_{>n}} X\pi_1 \to \hocolim_{\bI_{>n} \times \bI_{>n}}X\pi_1$$ is a weak equivalence. By composition, we thus have a weak equivalence $$\hocolim_{\{d\} \times \bI_{>n}} X\pi_1 \to \hocolim_{\bI_{>n}}X$$ From \cref{bar construction const functor}, we have a homeomorphism $$\hocolim_{\{d\} \times \bI_{>n}} X\pi_1 \cong X(d) \times \cB\bI_{>n}$$ As $\cB\bI_{>n}$ is contractible, we have a homotopy equivalence 
$$X(d) \to \hocolim_{\{d\} \times \bI_{>n}} X\pi_1$$ These give a weak equivalence $X(d) \to \hocolim_{\bI_{>n}}X$.
\end{proof}


\section{Infinite loop space structures on \texorpdfstring{$\cB^2\cG_A, \cB^2P\cG_A$}{B2GA,B2PGA} and \texorpdfstring{$\cB^2\wcG_A$}{B2tildeGA}}\label{section 4}
In this section we will prove our main result by constructing $\bI$-FCPs associated to the monoidal categories $\cG_A$, $P\cG_A$ and $\wcG_A$. We will use the notation that we introduced in \cref{strict monoidal cats from crossed modules}.

Let $A$ be a separable SSA $C^*$-algebra with $U(A)$ connected. As $A$ is SSA, it follows that $U(A^{\otimes m})$ is isomorphic to $U(A)$. As $U(A)$ is connected, $U(A^{\otimes m})$ has a universal cover for all $m \geq 1$. Our preferred model for this was described in detail in \cref{universal-cover of group}. Let $p_m \colon \wcU(A^{\otimes m}) \to U(A^{\otimes m})$ be the covering map.

Since $A$ is simple, we have $ZU(A) \cong U(1)$. Let $q_m \colon U(A^{\otimes m}) \to PU(A^{\otimes m})$ be the quotient map which identifies $PU(A^{\otimes m})$ with $U(A^{\otimes m})/U(\C^{\otimes m})$. We denote $q_m(u)$ for $u \in U(A^{\otimes m})$, by $[u]$.

Depending on the context we will use the notation $1_A$ for the multiplicative unit as well as the identity morphism $A \to A$, which we might also denote by $\id_A$. We begin by constructing functors $\bI \to \TCat$. But first, we need a preliminary remark.

\begin{rem}\label{I-morphism factorization}
    Let $f \colon [n] \to [m]$ be an injection in $\mor{\bI}$. In particular, we have $n \leq m$. Let $\iota_{n,m} \colon [n] \to [m]$ be the inclusion map. There exists a bijection $\sigma_f \colon [m]\to [m]$ in $\mor{\bI}$ such that ${f = \sigma_f \circ \iota_{n,m}}$.
\end{rem}

The idea now is to produce functors out of $\sigma_f$ and $\iota_{n, m}$.

\begin{const}\label{helpers for F_A}
Any bijection $\sigma \colon [m] \to [m]$ in $\bI$ induces an isomorphism $\overline{\sigma} \colon A^{\otimes m} \to A^{\otimes m}$ via ${a_1 \otimes a_2 \otimes ... \otimes a_m \mapsto a_{\sigma(1)} \otimes a_{\sigma(2)} \otimes ... \otimes a_{\sigma(m)}}$ and then using the density of the linear span of the simple tensors.
        
This isomorphism in turn induces a functor $\Sigma_{\sigma} \colon \cG_{A^{\otimes m}} \to \cG_{A^{\otimes m}}$ given by 
\begin{equation} \label{eqn:GA_Sigma}
    \begin{tikzcd}[cramped]
	\alpha & {\overline{\sigma} \circ \alpha \circ \overline{\sigma}^{-1}} \\
	{\Ad_u \alpha} & {\Ad_{\overline{\sigma}(u)} \overline{\sigma} \circ \alpha \circ \overline{\sigma}^{-1}} 
	\arrow[maps to, from=1-1, to=1-2]
	\arrow["{(u, \alpha)}"', from=1-1, to=2-1]
	\arrow["{(\overline{\sigma}(u), \overline{\sigma} \circ \alpha \circ \overline{\sigma}^{-1})}", from=1-2, to=2-2]
	\arrow[maps to, from=2-1, to=2-2]
    \end{tikzcd}
\end{equation}
for $\alpha \in \Aut(A) = \ob{\cG_A}$ and $(u, \alpha) \in \mor{\cG_A}$. 

If $n \leq m$, we have a functor $\mathcal{I}_{n, m} \colon \cG_{A^{\otimes n}} \to \cG_{A^{\otimes m}}$ given by 
\begin{equation} \label{eqn:GA_Inm}
    \begin{tikzcd}[cramped]
	\alpha & {\alpha \otimes \id_{A^{ \otimes (m - n)}}} \\
	{\Ad_{u}\alpha} & {\Ad_{u \otimes 1_{A^{ \otimes (m - n)}}}\alpha \otimes \id_{A^{ \otimes (m - n)}}}
	\arrow[maps to, from=1-1, to=1-2]
	\arrow["{(u, \alpha)}"', from=1-1, to=2-1]
	\arrow["{(u \otimes 1_{A^{ \otimes (m - n)}}, \alpha \otimes id_{A^{ \otimes (m - n)}}) }", from=1-2, to=2-2]
	\arrow[maps to, from=2-1, to=2-2]
    \end{tikzcd}
\end{equation}
for $\alpha \in \Aut(A) = \ob{\cG_A}$ and $(u, \alpha) \in \mor{\cG_A}$. In both cases it is straightforward to check that the constructions are compatible with composition.

There are analogous constructions of $\mathcal{I}_{n, m}$ and $\Sigma_{\sigma}$ for $P\cG_A$ and $\widetilde{G}_{A}$ as well. In fact, since $A$ is SSA we have $PU(A) = U(A)/U(\C)$. Hence, to get the diagram for $P\cG_A$ one just needs to replace the unitaries in both diagram \eqref{eqn:GA_Inm} and \eqref{eqn:GA_Sigma} by their equivalence classes in $PU(A)$. Note that $\Ad_{\overline{\sigma}(u)} = \Ad_{\overline{\sigma}([u])}$, $\Ad_{u \otimes 1} = \Ad_{[u \otimes 1]}$, $[\sigma(u)] = \sigma([u])$ and $[u \otimes 1]$ only depends on $[u]$. We will denote the resulting functors by $P\Sigma_{\sigma}$ and $P\mathcal{I}_{n,m}$.

To obtain the analogous construction for $\wcG_A$ we lift the functors for $\cG_A$. A bijection $\sigma \colon [m] \to [m]$ defines a $*$-isomorphism $\overline{\sigma} \colon A^{\otimes m } \to A^{\otimes m}$. Restricting to $U(A^{\otimes m})$ it induces a topological group isomorphism $\overline{\sigma} \colon U(A^{\otimes m}) \to U(A^{\otimes m})$. It is easy to see (using \cref{functorial construction of universal cover}) that $\overline{\sigma} \colon U(A^{\otimes m}) \to U(A^{\otimes m})$ then defines a topological group isomorphism $\widetilde{\sigma} \colon \wcU(A^{\otimes m}) \to \wcU(A^{\otimes m})$ with $p_m \circ \widetilde{\sigma} = \overline{\sigma} \circ p_m$.
    
Therefore define a functor $\widetilde{\Sigma}_{\sigma} \colon \wcG_{A^{\otimes m}} \to\wcG_{A^{\otimes m}}$ via:
\[
\begin{tikzcd}[cramped]
	\alpha & {\overline{\sigma} \circ \alpha \circ \overline{\sigma}^{-1}} \\
	{\Ad_{p_m(\widetilde{u})}\alpha} & {\Ad_{p_m(\widetilde{\sigma}(u))}\overline{\sigma} \circ \alpha \circ \overline{\sigma}^{-1}}
	\arrow[maps to, from=1-1, to=1-2]
	\arrow["{{(\widetilde{u}, \alpha)}}"', from=1-1, to=2-1]
	\arrow["{{(\widetilde{\sigma}(u), \overline{\sigma} \circ \alpha \circ \overline{\sigma}^{-1})}}", from=1-2, to=2-2]
	\arrow[maps to, from=2-1, to=2-2]
\end{tikzcd}
\] 
The equation $p_m \circ \widetilde{\sigma} = \overline{\sigma} \circ p_m$ implies that this is indeed well-defined and a functor.

Similarly, if $n \leq m$, one has the $*$-homomorphism $A^{\otimes n} \to A^{\otimes m}$ which is defined on the simple tensors via $a_1 \otimes ... \otimes a_n \mapsto a_1 \otimes ... \otimes a_n \otimes 1^{\otimes (m-n)}$ and extended using density. It induces a $*$-homomorphism $i_{n, m} \colon U(A^{\otimes n}) \to U(A^{\otimes m})$ which, by \cref{functorial construction of universal cover}, induces a group homomorphism $\widetilde{i}_{n, m}\colon \wcU(A^{\otimes n}) \to \wcU(A^{\otimes m})$ with $p_m \circ \widetilde{i}_{n, m} = i_{n, m} \circ p_n$.

Hence, we can define a functor $\widetilde{\mathcal{I}}_{n, m} \colon \wcG_{A^{\otimes n}} \to  \wcG_{A^{\otimes m}}$ given by:
\[\begin{tikzcd}[cramped]
	\alpha & {\alpha \otimes \id_{A^{ \otimes (m - n)}}} \\
	{\Ad_{p_n(\widetilde{u})}\alpha} & {\Ad_{p_m(\widetilde{i}_{n,m}(\widetilde{u}))}\alpha \otimes \id_{A^{ \otimes (m - n)}}}
	\arrow[maps to, from=1-1, to=1-2]
	\arrow["{{(\widetilde{u}\, \alpha)}}"', from=1-1, to=2-1]
	\arrow["{{(\widetilde{i}_{n, m}(\widetilde{u}), \alpha \otimes id_{A^{ \otimes (m - n)}}) }}", from=1-2, to=2-2]
	\arrow[maps to, from=2-1, to=2-2]
\end{tikzcd}\]
\end{const}

We are now in the position to define our three $\bI$-FCPs.

\begin{defn}
    We define functors $F_A, PF_A, \wcF_{A} \colon \bI \to \TCatmon$ from $\bI$ to the category of topological strict monoidal categories as follows:
    \begin{itemize}
        \item For $[n] \in \ob{\bI}$ with $n > 0$, we set 
        \begin{align*}
            F_A([n]) &:= \cG_{A^{\otimes n}}\ ,\\
            PF_A([n]) &:= P\cG_{A^{\otimes n}}\ ,\\
            \wcF_{A}([n]) &:= \wcG_{A^{\otimes n}}
        \end{align*} and $[0]$ always maps to the category with one object and one morphism.
        \item For $f \in \bI([n], [m])$, we write $f = \sigma_f \circ \iota_{n, m}$ as in \cref{I-morphism factorization} and set:
                \begin{align*}
            F_A(f) &:= \Sigma_{\sigma_f} \circ \mathcal{I}_{n, m}\ ,\\
            PF_A(f) &:= P\Sigma_{\sigma_f} \circ P\mathcal{I}_{n, m}\ ,\\
            \wcF_{A}(f) &:= \widetilde{\Sigma}_{\sigma_f} \circ \widetilde{\mathcal{I}}_{n, m}\ .
        \end{align*}
    \end{itemize}
\end{defn}

Since the decomposition $f = \sigma_f \circ \iota_{n,m}$ is not unique, we need to check that these definitions make sense.  

\begin{prop}
    $F_A, PF_A, \wcF_A$ are well-defined and indeed functors.
\end{prop}
\begin{proof}
    We will provide the full proof for $F_A$. The other cases will follow using minor modifications of this argument by either projecting down to $PU(A^{\otimes m})$ or lifting to $\wcU(A^{\otimes m})$ using the functor $G\mapsto \widetilde{G}$ (see \cref{universal-cover of group}).
    
    \begin{claim*}
    $F_A$ is well-defined and takes values in $\TCatmon$.
    \end{claim*}
    Let $f = \sigma_f' \circ \iota_{n, m} = \sigma_f \circ \iota_{n,m}$ with $\sigma_f$ and $\sigma_f'$ bijections $[m] \to [m]$. Notice that this implies that $\sigma_f$ and $\sigma_f'$ have to agree on the first $n$ elements of $[m]$, while the elements $n+1, \dots, m$ are mapped bijectively to the gaps between $\{\sigma_f(1), \dots, \sigma_f(n)\} \subset [m]$ in an order that might differ for $\sigma_f$ and $\sigma_f'$. 
        
    We need to show the equality of functors $\Sigma_{\sigma_f} \circ \mathcal{I}_{n,m} = \Sigma_{\sigma'_f} \circ \mathcal{I}_{n,m}$. Both, on objects and on morphisms, the functor $\mathcal{I}_{n,m}$ appends identities (either tensor copies of $1_A$ or of $\id_A$) at the end, i.e.\ in the positions $n+1, \dots, m$. By our considerations above, $\Sigma_{\sigma_f}$ and $\Sigma_{\sigma_f'}$ will permute these identities into the gaps between $\sigma_f(1), \dots, \sigma_f(n)$. But since the elements that are permuted are all the same, the order in which they get sorted into the gaps does not matter. Hence, we have $\Sigma_{\sigma_f} \circ \mathcal{I}_{n,m} = \Sigma_{\sigma'_f} \circ \mathcal{I}_{n,m}$.

    It remains to be seen that $\Sigma_{\sigma}$ and $\mathcal{I}_{n,m}$ are compatible with the monoidal structures $\ostar$ on $\cG_{A^{\otimes n}}$ (see \cref{def: topological strict monoidal cat associated to topological crossed module}). Since conjugation by $\overline{\sigma}$ and tensoring with $\id$ are group homomorphisms, this is clear on objects. On morphisms we have
    \begin{align*}
        \Sigma_{\sigma}(u,\alpha) \ostar \Sigma_{\sigma}(v,\beta) & = (\overline{\sigma}(u), \overline{\sigma}\alpha\overline{\sigma}^{-1}) \otimes (\overline{\sigma}(v), \overline{\sigma}\beta\overline{\sigma}^{-1}) = (\overline{\sigma}(u\alpha(v)), \overline{\sigma}\alpha\beta\overline{\sigma}^{-1}) \\
        & = \Sigma_{\sigma}(u\alpha(v), \alpha\beta) = \Sigma_{\sigma}( (u, \alpha) \ostar (v,\beta))\ ,\\
        \mathcal{I}_{n,m}(u,\alpha) \ostar \mathcal{I}_{n,m}(v,\beta)  & =  (u \otimes 1, \alpha \otimes \id) \otimes (v \otimes 1, \beta \otimes \id) = (u\alpha(v) \otimes 1, \alpha\beta \otimes \id) \\
        & = \mathcal{I}_{n,m}(u\alpha(v),\alpha\beta) = \mathcal{I}_{n,m}((u, \alpha) \ostar§ (v,\beta))\ .
    \end{align*}

    \begin{claim*}
        $F_A$ is a functor.
    \end{claim*}
        It is clear that $F(\id_{[m]}) = \id_{\cG_{A^{\otimes m}}}$.
    Let $[n] \xrightarrow{f} [m] \xrightarrow{g} [k]$ be morphisms in $\bI$, in particular $n \leq m \leq k$. We have decompositions $f = \sigma_f \circ \iota_{n,m}$ and $g = \sigma_g \circ \iota_{m, k}$ where $\sigma_{f} \colon [m] \to [m]$ and $\sigma_{g} \colon [k] \to [k]$ are bijections. Then, $g \circ f = \sigma_g \circ \iota_{m,k} \circ \sigma_{f} \circ \iota_{n, m}$.

    Let $\sigma \colon [k] \to [k]$ be the permutation that acts like $\sigma_f$ on the first $m$ elements and is the identity on $m+1, \dots, k$. Then we have 
    \[
        \iota_{m, k} \circ \sigma_f = \sigma \circ \iota_{m, k} \qquad \text{and} \qquad \overline{\sigma_f} \otimes \id_{A^{\otimes (k-m)}} = \overline{\sigma}\ ,
    \]
    which gives us $g \circ f = \sigma_g \circ \sigma \circ \iota_{n, k}$ and $\mathcal{I}_{m,k} \circ \Sigma_{\sigma_f} = \Sigma_{\sigma} \circ \mathcal{I}_{m,k}$. Therefore
    \[
        F_A(g) \circ F_A(f) = \Sigma_{\sigma_g} \circ \mathcal{I}_{m,k} \circ \Sigma_{\sigma_f} \circ \mathcal{I}_{n,m} = \Sigma_{\sigma_{g} \circ \sigma} \circ \mathcal{I}_{n, k} = F_A(g \circ f)\ ,
    \]
    which finishes the proof.
\end{proof}

Our three functors take values in monoidal categories. Therefore we can apply $\cB^2$ pointwise to get three functors $\cB^2F_A, \cB^2PF_A$ and $\cB^2\wcF_A$ from $\bI$ to $\Top$. We now check that these satify the conditions in \cref{equiv characterization of commutative monoid}.

\begin{prop}\label{monoids under day conv for our functors}
    $\cB^2F_A, \cB^2PF_A, \cB^2\wcF_A$ are $\bI$-FCPs.
\end{prop}
\begin{proof}
    For $m, n > 0$, one has the functors $\mu_{n,m}:\cG_{A^{\otimes m}} \times \cG_{A^{\otimes n}} \to \cG_{A^{\otimes (m+n)}}$ given by $(\alpha, \beta) \mapsto \alpha \otimes \beta$ on objects and $((u,\alpha), (v, \beta)) \mapsto (u \otimes v, \alpha \otimes \beta)$ on morphisms. For any $k \geq 0$, we define $\mu_{k, 0} = \pi_{1}$ and $\mu_{0, k} = \pi_2$ ($\pi_i$ is the projection to the $i$th factor). These make diagrams similar to the ones in \cref{equiv characterization of commutative monoid} commute and turn $F_A$ into an FCP valued in $\TCat$. 
    
We verify that each $\mu_{m, n}$ ($m, n \geq 0$) is a morphism of monoidal categories. Let $\ostar_n \colon \cG_{A^{\otimes n}} \times \cG_{A^{\otimes n}} \to \cG_{A^{\otimes n}}$ be the monoidal structure on $\cG_{A^{\otimes n}}$ as described in \cref{def: topological strict monoidal cat associated to topological crossed module}. We need to check the commutativity of the following diagram
\[
\begin{tikzcd}[cramped]
	{(\cG_{A^{\otimes n}} \times \cG_{A^{\otimes m}}) \times (\cG_{A^{\otimes n}} \times \cG_{A^{\otimes m}})} && {\cG_{A^{\otimes (n+m)}} \times \cG_{A^{\otimes (n+m)}}} \\
	{\cG_{A^{\otimes n}} \times \cG_{A^{\otimes m}}} && {\cG_{A^{\otimes (n+m)}}}
	\arrow["{(\mu_{n,m}, \mu_{n,m})}", from=1-1, to=1-3]
	\arrow["{(\ostar_n, \ostar_m) \circ \tau}"', from=1-1, to=2-1]
	\arrow["{\ostar_{n+m}}", from=1-3, to=2-3]
	\arrow["{\mu_{n,m}}"', from=2-1, to=2-3]
\end{tikzcd}
\] 
where $\tau$ is the functor permuting the two middle factors.

As $\cG_{A^{\otimes 0}}$ is, by definition, the trivial category, the commutativity of the above diagram is apparent for either $m = 0$ or $n = 0$. For $m, m > 0$, let $(u, \alpha), (w, \gamma)$ be arrows in $\cG_{A^{\otimes n}}$ and let $(v, \beta)$ and $(x, \delta)$ be arrows in $\cG_{A^{\otimes m}}$. This follows from the Eckmann-Hilton-like equation
\begin{align*}
   & ((u,\alpha) \otimes (v,\beta)) \ostar ((w,\gamma) \otimes (x,\delta)) = (u \otimes v, \alpha \otimes \beta) \ostar (w \otimes x, \gamma \otimes \delta) \\
   =& (u\alpha(w) \otimes v\beta(x), \alpha\gamma \otimes \beta\delta) = ((u,\alpha) \ostar (w,\gamma)) \otimes ((v, \beta) \ostar (x,\delta))\ .
\end{align*}

Note that $U(A^{\otimes k}), \Aut(A^{\otimes k})$ are Hausdorff $k$-groups by \cref{main groups are all hausdorff k-groups} and hence, so is $\cB\cG_{A^{\otimes k}}$ by \cref{weak equiv between pennig constr and B^2} and \cref{crossed module cat gives top group after realization}. Since $\mu_{n,m}$ are monoidal functors for all $m, n \geq 0$, they give group homomorphisms $\cB\cG_{A^{\otimes m}} \times \cB\cG_{A^{\otimes n}} \to \cB\cG_{A^{\otimes (m+n)}}$. Applying $\cB$ again and using \cref{B commutes with products} we get maps 
\[
    \cB^2F_A([m]) \times \cB^2F_A([n]) \to \cB^2F_A([m + n]) = \cB^2F_A([m] \oplus [n])
\] 
which constitute the natural transformations required in \cref{equiv characterization of commutative monoid}.

Similarly, one applies the same idea (noting that $PU(A^{\otimes k})$ and $\wcU(A^{\otimes k})$ are also Hausdorff $k$-groups by \cref{main groups are all hausdorff k-groups}) to the functors $P\cG_{A^{\otimes m}} \times P\cG_{A^{\otimes n}} \to P\cG_{A^{\otimes (m+n)}}$ given by $\alpha, \beta \mapsto \alpha \otimes \beta$ on objects and $(([u], \alpha), ([v], \beta)) \mapsto ([u \otimes v], \alpha \otimes \beta)$ on morphisms (for $PF_A$) and to the functors $\widetilde{G}_{A^{\otimes m}} \times \widetilde{G}_{A^{\otimes n}} \to \widetilde{G}_{A^{\otimes (m+n)}}$ given by $\alpha, \beta \mapsto \alpha \otimes \beta$ on objects and $((\widetilde{u}, \alpha), (\widetilde{v}, \beta)) \mapsto (\widetilde{u} \hat{\otimes} \widetilde{v}, \alpha \otimes \beta)$ on morphisms (for $\wcF_A$) where $\hat{\otimes} \colon \wcU(A^{\otimes m}) \times \wcU(A^{\otimes n}) \to \wcU(A^{\otimes (m + n)})$ is the group homomorphism induced by the group homomorphism $U(A^{\otimes m}) \times U(A^{\otimes n}) \to U(A^{\otimes (m + n)}) : x, y \mapsto x \otimes y$ after application of the functor $G \mapsto \widetilde{G}$ (see \cref{universal-cover of group}).
\end{proof}

We show that $\cB^2F_A, \cB^2P F_A, \cB^2\wcF_A$ satisfy the hypotheses of B\"okstedt's lemma.

\begin{prop}\label{weak equiv for our functors}
    $\cB^2F_A(f), \cB^2PF_A(f)$ and $\cB^2\wcF_A(f)$ are weak equivalences for all $f \in \bI([m], [n])$ with $m, n > 0$.
\end{prop}
\begin{proof}
    This is where the SSA hypothesis on $A$ becomes relevant. For brevity, let $X = F_A$, $PF_A$ or $\wcF_A$. Writing $f = \sigma_{f} \circ \iota_{n, m}$ as in \cref{I-morphism factorization}, it is enough to check that $\cB^2X(\iota_{n, m})$ is a weak equivalence for all $m \geq n > 0$. Because $\iota_{n,m}$ factors into $\iota_{k, k+1}$ for $k = n,\dots, m-1$, this reduces further to showing that $\iota_{n, n+1}$ induces weak equivalences $\cB^2X(\iota_{n, n+1})$ for all $n > 0$. By \cref{pointwise weak equiv realize to weak equiv}, it suffices to see that $\iota_{n, n+1}$ induces a weak equivalence $\cB X(\iota_{n, n+1}) \colon \cB X([n]) \to \cB X([n+1])$. Using \cref{lem:nerve_of_crossed_mod_and_bar_construction} and \cref{big topological lemma for ua and pua}, it is easy to see that $NX([n])$ is a Reedy cofibrant simplicial space for all $n \geq 0$. Since $\cB X(\iota_{n, n+1})$ is induced by a simplicial map $NX([n]) \to NX([n+1])$, it suffices by \cref{pointwise weak equiv realize to weak equiv} to check that this simplicial map is a point-wise weak equivalence. For $m > 0$ and using \cref{lem:nerve_of_crossed_mod_and_bar_construction}, we identify $N_mX([n]) \to N_mX[n+1]$ with:
    \begin{align*}
        \Aut(A^{\otimes n}) \times U(A^{\otimes n})^{m} \times \ast &\to \Aut(A^{\otimes (n+1)}) \times U(A^{\otimes (n+1)})^{m} \times \ast\\
        (x, (h_1, h_2, ..., h_m), \ast) &\mapsto (x \otimes \id_A, (h_1 \otimes 1_A, h_2 \otimes 1_A, ..., h_m \otimes 1_A), \ast)
    \end{align*} if $X = F_A$. For $X = PF_A$ we replace $h_j$ by $[h_j]$ and $h_j \otimes 1$ by $[h_j \otimes 1]$. In case of $X = \wcF_A$ we use $\widetilde{i}_{n,n+1}(h_j)$ (see \cref{helpers for F_A} for the definition of $\widetilde{i}_{n,n+1}$) in place of $h_j \otimes 1$.

    For $m = 0$, the contractibility of $\Aut(A^{\otimes n})$ for an SSA $C^*$-algebra $A$ proven in \cite[Theorem~2.3]{Dadarlat_2015} implies that $N_0X[n] \to N_0X[n+1]$ is a weak equivalence. For $m > 0$ and $X = F_A$ or $PF_A$, it follows from \cref{big topological lemma for ua and pua} (along with the *-isomorphism $A^{\otimes n} \cong A$) that $N_mX[n] \to N_mX[n+1]$ is a weak equivalence. For $m > 0$ and $X = \wcF_A$, it follows by the long exact sequence of homotopy groups for the universal coverings $\wcU(A^{\otimes n}) \to U(A^{\otimes n})$ and the five-lemma that $N_mX[n] \to N_mX[n+1]$ is a weak equivalence.
\end{proof}

Thus, we have the following

\begin{thm} \label{thm:ssa_weak_htpy_type}
    Let $A$ be an SSA $C^*$-algebra with $U(A)$ connected (note that this condition is automatic if $A$ satisfies the UCT). Then
    \begin{itemize}
        \item there are weak equivalences \begin{align*}
    \cB^2\cG_A &= \cB^2F_A([1])\xrightarrow{\simeq} \hocolim_{\bI}\cB^2F_A\ ,\\
        \cB^2P\cG_A &= \cB^2PF_A([1])\xrightarrow{\simeq}\hocolim_{\bI}\cB^2PF_A\ ,\\
        \cB^2\wcG_A &= \cB^2\wcF_A([1])\xrightarrow{\simeq} \hocolim_{\bI}\cB^2\wcF_A\ .
    \end{align*}
    \item $\hocolim_{\bI}\cB^2F_A, \hocolim_{\bI}\cB^2PF_A$ and $\hocolim_{\bI}\cB^2\wcF_A$ are infinite loop spaces.
    \end{itemize}
        
           In particular, $\cB^2\cG_A, \cB^2P\cG_A$ and $\cB^2\wcG_A$ are infinite loop spaces.
\end{thm}

\begin{proof}
    For brevity, let $Y := \cB^2F_A$ or $\cB^2PF_A$ or $\cB^2\wcF_A$ and let $X := \cB^2\cG_A$ or $\cB^2P\cG_A$ or $\cB^2\wcG_A$ accordingly. Let us also write $hY := \hocolim_{\bI}Y$.

    From \cref{weak equiv for our functors} and \cref{boeksted lemma ref}, it follows that the canonical map $X = Y([1]) \to hY$ is a weak equivalence. As $X$ is path-connected, the weak equivalence implies $\pi_0(hY) = \ast$, the trivial group.
    We also showed that $Y$ is a $\bI$-FCP. By \cref{Linds gamma space machine} we get a $\Gamma$-space $R_Y \colon \Gamma^{\op} \to \Top$ such that $R_Y(1^+) = hY$ and $R_Y(0^+) = \ast$. By \cref{inf loop space machine used}, $hY$ is weakly equivalent to an infinite loop space. Hence, $X$ is an infinite loop space.
\end{proof}

We briefly recall one of the main consequences of the above result. Let $\Gamma$ be a discrete group. It was shown in \cite{pacheco2025gkernelscrossedmodules} that $\Gamma$-kernels and cocycle actions up to their respective notions of conjugacy are in bijection with cohomology sets $H^1(\Gamma,P\cG_A)$ and $H^1(\Gamma,\cG_A)$, respectively. Moreover, there is a natural tranformation 
\[
    H^1(\Gamma, \cG) \to [\cB\Gamma, \cB^D\cG]
\]
from the cohomology set to a homotopy set. By \cref{thm:ssa_weak_htpy_type} the right hand side of this transformation actually has much more structure in the case of SSA algebras.
\begin{cor}
    Let $A$ be an SSA $C^*$-algebra. Let $\cG$ be one of the associated topological crossed modules $\cG_A$, $P\cG_A$ or $\wcG_A$ (where $U(A)$ is assumed to be connected in the final case). Let $\Gamma$ be a discrete group. Then
    \[
        H^1(\Gamma, \cG) \to [\cB\Gamma, \cB^D\cG] \cong [\cB\Gamma,\cB^2\cG]
    \]
    takes values in an abelian group that is part of a cohomology theory.
\end{cor}

\begin{rem}
    In follow-up work we will study natural transformations between the resulting cohomology theories, in particular the ones induced by the infinite loop space maps 
    \[
        \cB^D\cG_A \to \cB^D P\cG_A \qquad , \qquad \cB^D\wcG_A \to \cB^D\cG_A \qquad \text{and} \qquad \cB^D\wcG_A \to \cB^D P\cG_A
    \]
    and how they relate to lifting obstructions.
\end{rem}


\section{Appendix (Proofs of lemmas)}\label{section 5 - appendix}

\begin{proof}[\textbf{Proof of \cref{big topological lemma for ua and pua}}]

In preparation of the proof, let us fix the following conventions:
\begin{itemize}
    \item Write $f \colon \mathbb{R} \to S^1$ for $x \mapsto \exp(ix)$.
    \item Fix $\pi > \epsilon > 0$ and let $\rho := \norm{e^{i(\pi-\epsilon)} - 1}$, $\delta := \norm{e^{i \frac{\pi - \epsilon}{2}}- 1}$.
    \item Let $\theta \colon f((-\pi+\epsilon, \pi-\epsilon)) \to \mathbb{R}$ be the unique function such that $\theta(e^{it}) = t$ for all $t \in (-\pi+\epsilon, \pi-\epsilon)$. 
    \item By functional calculus, $f$ and $\theta$ induce continuous maps from a subset of $A_{sa}$ to $A$ which are inverses.
    \item  Let $V := \{x \in U(A) : \norm{x - 1} < \rho\}$.
    \item  Let $R := \{h \in A_{sa} : \norm{h} < \pi-\epsilon\}$.
    \item Let $R' := \{h \in A_{sa} : \norm{h} < \frac{\pi - \epsilon}{2}\} \subset R$.
    \item Let $V' := \{x \in U(A) : \norm{x - 1} < \delta\} \subset V$
\end{itemize}

It is clear that $f$ maps $V$ to $R$ and, so, its inverse $\theta$ maps $R$ to $V$.

\begin{proof}[\textbf{Proof of \cref{big topological lemma for ua and pua}(1)}]

Note that $\theta(1) = 0$. It follows that $V$ strongly deformation retracts onto $1 \in U(A)$ via the homotopy 
\[
H \colon V \times I \to U(A) : (x, t) \mapsto e^{i(1-t)\theta(x)}\ .
\]
We have the continuous function $\psi \colon U(A) \to I : x \mapsto \min\{\frac{\norm{x-1}}{\rho}, 1\}$ with $\psi^{-1}(0) = 1$ and $\psi \equiv 1$ on $X - V$. By \cite[Theorem~2]{stromcofib1}, $(U(A), 1)$ is a cofibered pair.
\end{proof}

\begin{proof}[\textbf{Proof of \cref{big topological lemma for ua and pua}(2)}]
 Let $q \colon U(A) \to PU(A) = U(A)/U(1)$ be the quotient map.  We apply \cite[Theorem~2]{stromcofib1} to two maps $H' \colon q(V') \times I \to PU(A)$ and $\hat{\sigma} \colon PU(A) \to I$ to be defined.  Firstly, note that $q$ is an open map because $U(1)$ acts continuously on $U(A)$ by left multiplication. Thus, $q(V')$ is open in $PU(A)$.

The map $H'$ is obtained in the following way. For $([x], t) \in q(V') \times I$, let $x \in V'$ be a representative of $[x]$ in $V'$ (note that $q^{-1}(q(V')) = U(1)V'$, so we really are making a choice). We define ${H'([x], t) = [e^{i(1-t)\theta(x)}]}$.

 \begin{claim*}
     This is well-defined and continuous.
 \end{claim*}
 \begin{proof}

     We have the homotopy $G \colon U(1) \times V' \times I \to PU(A) : (\lambda, x, t) \mapsto [e^{it\theta(x)}]$. The multiplication $\mu \colon U(1) \times U(A) \to U(A)$ is clearly an open map. Because $V'$ is open in $U(A)$, we have that the restriction $\mu \colon U(1) \times V' \to U(1)V'$ is a quotient map. Now, let $x, x' \in V'$ and $\lambda, \lambda' \in U(1)$ such that $\lambda x= \lambda'x'$. Then,  $x$ and $x'$ commute and $x(x')^\ast = z := \lambda^{\ast}\lambda' \in U(1)$. As $\theta(x), \theta(x') \in R'$, we have that $\norm{\theta(x) - \theta(x')} < \pi - \epsilon < \pi$. From the facts that $xx' = x'x$ and that $x, x' \in V'$ we get that $\theta(x)$ and $\theta(x')$ commute (this follows by a quick application of the Stone-Weierstrass theorem). By the commutativity of $\theta(x)$ and $\theta(x')$ we have that $e^{i(\theta(x)-\theta(x'))} = x(x')^{\ast} = z$. Let $\zeta \in [0, 2\pi)$ be such that $z = e^{i\zeta}$. It follows by the spectral mapping theorem that $\spec(\theta(x) - \theta(x')) \subseteq 2\pi\mathbb{Z} + \zeta$. As $\theta(x) - \theta(x')$ is self-adjoint and $\norm{\theta(x) - \theta(x')} < \pi$ and $\spec(\theta(x) - \theta(x')) \subseteq 2\pi \mathbb{Z} + \zeta$, it follows by the fact that $\zeta \in [0, 2\pi)$ that the spectrum of $\theta(x) - \theta(x')$ is just a point. Hence, $\theta(x) - \theta(x') = 2\pi k + \zeta \in (-\pi, \pi)$ for some $k \in \{-1, 0\}$. Hence, for all $t \in I$ we have that $e^{it\theta(x)} = e^{it(2\pi k + \zeta) + it\theta(x')} = z_te^{it\theta(x')}$ for some $z_t \in U(1)$ which implies that $[e^{it\theta(x)}] = [e^{it\theta(x')}]$ in $PU(A)$. Thus, $G(\lambda', x', t) = G(\lambda, x, t)$ for all $t \in I$ and all $\lambda, \lambda' \in U(1)$ and $x, x' \in V'$ such that $\lambda x = \lambda' x'$.  By properties of quotients, it follows that we have a continuous map $\hat{G} \colon U(1)V' \times I \to PU(A)$ that extends $G$ along $\mu$. Moreover, for any $x \in V'$ and any $\lambda, \lambda' \in U(1)$ we have that $\hat{G}(\lambda x, t) = \hat{G}(\lambda'x, t)$. It follows again by properties of quotients, that there is a (unique) continuous map $q(V') \times I \to PU(A)$ which extends $\hat{G}$ along $q$. This map coincides with $H'$.
 \end{proof} 
 The map $\hat{\sigma}$ is obtained in the following way. For $[x] \in PU(A)$, let $x \in U(A)$ be a representative of $[x]$. We define $\hat{\sigma}([x]) = \min\{\frac{\text{dist}(x, U(1))}{\delta}, 1\}$.

 \begin{claim*}
     This is well-defined and continuous.
 \end{claim*}
 \begin{proof}
     Let us consider the function $\sigma \colon U(A) \to [0, 1] : x \mapsto \min\{\frac{\text{dist}(x, U(1))}{\delta}, 1\}$. By the closedness of $U(1)$, we have $\sigma^{-1}(0) = U(1)$. Now, let $x \in U(A) - U(1)V'$.     It follows that for all $\lambda \in U(1)$, we have that $\lambda x \notin V'$. Hence, $\norm{\lambda x - 1} \geq \delta$ which implies that $\norm{x - \lambda^\ast} \geq \delta$ for all $\lambda \in U(1)$. This implies that $\text{dist}(x, U(1)) \geq \delta$ and hence $\sigma(x) = 1$. That is, $\sigma \equiv 1$ in $U(A) - U(1)V'$. Moreover, $\sigma(\lambda x) = \sigma(\lambda' x)$ for all $\lambda \in U(1)$ and all $x \in U(A)$. By properties of quotients, it follows that there is a (unique) continuous function $\hat{\sigma} \colon PU(A) \to [0, 1]$ which extends $\sigma$ along $q$. This function is precisely $\hat{\sigma}$.
 \end{proof}

It is evident that $\hat{\sigma}^{-1}(0) = \{[1]\} \in PU(A)$ and $\hat{\sigma} \equiv 1$ in $PU(A) - q(V')$, which concludes the proof of \cref{big topological lemma for ua and pua} (2).
\end{proof}

\begin{proof}[\textbf{Proof of \cref{big topological lemma for ua and pua}(3)}]
   Because $U(1)$ is normal in $U(A)$, we just need to show that $q \colon U(A) \to PU(A)$ has a section in a neighborhood of $[1_A]$. Consider the function $\tau \colon V' \to U(1) : x \mapsto e^{i \cdot \max(\spec(\theta(x)))}$.  We define the section $f \colon q(V') \to U(A)$ by $f([x]) = \tau(x)^{-1}x$ where $x \in V'$ is a representative of $[x]$.
    \begin{claim*}
        This is well-defined and continuous
    \end{claim*}
    \begin{proof}
           The inequality $|\max(\spec(x)) - \max(\spec(y))| \leq \norm{x-y}$, for all $x=x^*, y=y^*$ in $A$, implies that $\tau$ is continuous. The maps $\mu \colon U(1) \times V' \to U(1)V' : (\lambda, x) \mapsto \lambda x$ and $q:U(1)V' \to q(V')$ (used in the previous proofs) are quotient maps. We define $t \colon U(1) \times V' \to U(1)V'$ via $(\lambda, x) \mapsto \tau(x)^{-1}x$. If $\lambda x = \lambda' x'$, then in the above proof we showed that $\theta(x) = \theta(x') + c$ with $e^{ic} = \lambda^*\lambda'$. It follows that $\tau(x) = \lambda^*\lambda'\tau(x')$. Therefore, $\tau(x)^{-1}x = \tau(x')^{-1}x'$. Hence, there is a (unique) continuous map $R \colon U(1)V' \to U(A)$ which extends $t$ along $\mu$. By construction, $R(ax) = R(bx)$ for all $a, b \in U(1)$ and all $x \in V'$.
    By quotient properties, we get a (unique) continuous map $q(V') \to U(A)$ that extends $R$ along $q$. This map coincides with $f$. \end{proof} The result follows because one can then define sections over neighborhoods of any element of $PU(A)$ by translating the section $f$.
\end{proof}

\begin{proof}[\textbf{Proof of \cref{big topological lemma for ua and pua}(4)}]
 Let $\psi, u$ be as in \cref{defn of ssa algebra}. Because $U(A) \to U(A \otimes A)$ is a group homomorphism, it suffices to show that this map induces isomorphisms $\pi_n(U(A), 1) \to \pi_n(U(A \otimes A), 1 \otimes 1)$ for all $n \geq 0$. Let $[\beta] \in \pi_n(U(A), 1)$ be represented by $\beta \colon (S^n, z_0) \to (U(A), 1)$. By compactness of $S^n$, there is  a $t \in [0, 1)$ such that
\[
    {\norm{\psi(\beta(z)) - u_t(\beta(z) \otimes 1)u_t^{\ast}} < \rho}
\]
for all $z \in S^n$. This implies that the spectrum of $\psi(\beta(z)) \cdot (u_t(\beta(z) \otimes 1)u_t^{\ast})^{*}$ is contained in $f([-\pi + \epsilon, \pi - \epsilon])$ for all $z \in S^n$. By functional calculus, we have a continuous map $r \colon S^n \to (A \otimes A)_{sa}$ with $r(z) = \theta(\psi(\beta(z)) \cdot (u_t(\beta(z) \otimes 1)u_t^{\ast})^{*})$. Note that $r(z_0) = 0$. Then we have
$\psi(\beta(z)) \cdot (u_t(\beta(z) \otimes 1)u_t^{\ast})^{*} = e^{ir(z)}$. So, $\psi(\beta(z)) = e^{ir(z)}u_t(\beta(z) \otimes 1)u_t^{\ast}$ for all $z \in S^n$. This gives a pointed homotopy: $(z, s) \mapsto e^{isr(z)}u_{st}(\beta(z) \otimes 1)u_{st}^{\ast}$
    from $\psi \circ \beta$ to $u_0 (\beta \otimes 1)u_0^{\ast}$. It follows that 
    $[\psi \circ \beta] = [u_0 (\beta \otimes 1) u_0^{\ast}]$ in $\pi_n(U(A \otimes A), 1 \otimes 1)$. Hence, $[\beta \otimes 1] = [u_0^{\ast}\psi(\beta)u_0]$ in $\pi_n(U(A \otimes A), 1 \otimes 1)$ for all $[\beta]$ in $\pi_n(U(A), 1)$. Thus, $u \mapsto u \otimes 1$ is a weak equivalence.
\end{proof}

\begin{proof}[\textbf{Proof of \cref{big topological lemma for ua and pua}(5)}]
By (3) the maps $U(A) \to PU(A)$ and $U(A \otimes A) \to PU(A \otimes A)$ are both principal $U(1)$-bundles. Noting that principal bundles are fibrations, the statement follows from (4) and the long exact sequence of homotopy groups applied to the commutative diagram
\[
    \begin{tikzcd}[/tikz/baseline=(tikz@f@2-2-1.base)]
        U(1) \ar[r] \arrow[equal]{d} & U(A) \ar[r] \ar[d,"u \mapsto u \otimes 1"] & PU(A) \ar[d,"{[u]} \mapsto {[u \otimes 1]}"] \\
        U(1) \ar[r] & U(A \otimes A) \ar[r] & PU(A \otimes A)
    \end{tikzcd} \qedhere
 \]
\end{proof} \noindent This completes the proof of \cref{big topological lemma for ua and pua} (5) and therefore the proof of the whole lemma.
\end{proof}

\begin{proof}[\textbf{Proof of \cref{weak equiv between pennig constr and B^2}}]
For a simplicial space $S$ let 
\[
    q \colon \coprod_{n \geq 0} S_n \times \Delta^{n} \to |S|
\] 
be the quotient map. Write $|S|^{n} := q(\coprod_{k \leq n} S_k\times \Delta^{k})$ and give this the quotient topology. First, let us note that the canonical map $S_0 \to |S|^{0}$ is a homeomorphism by \cite[Proposition~2.2(i)]{Schwede2026SpacesVsSimplicialSets}.
The main technical result is the following.

\begin{lemma*}
    Let $S$ be a Reedy cofibrant simplicial space. Then, $|S|$ is Hausdorff and $|S|^{0} \to |S|$ is a closed cofibration.
\end{lemma*}
\begin{proof}
We know that the inclusion of the latching object $L_{\ast}S \to S$ is a pointwise closed cofibration (see \cref{reedy cofib}). Recall that $|S|$ is the colimit of the $\{|S|^{n}\}_{n \geq 0}$ along the natural inclusions $|S|^{n} \to |S|^{n+1}$ and that there are pushouts 
\[
\begin{tikzcd}
	{L_{n+1}S \times \Delta^{n+1} \cup S_{n+1} \times \partial\Delta^{n+1}} & {|S|^{n}} \\
	{S_{n+1} \times \Delta^{n+1}} & {|S|^{n+1}}
	\arrow[from=1-1, to=1-2]
	\arrow[from=1-1, to=2-1]
	\arrow[from=1-2, to=2-2]
	\arrow[from=2-1, to=2-2]
\end{tikzcd} 
\]  
As $S$ is Reedy cofibrant, left vertical map is a closed cofibration. It follows that $|S|^{n} \to |S|^{n+1}$ is a closed cofibration because pushouts preserve closed cofibrations. 

By induction on $n$ (with $|S|^{0} \cong S_0$) and using the above pushout squares, it follows that each $|S|^{n}$ is Hausdorff. By \cite[Proposition~2.6]{Steenrod1967}, $|S|^n$ is a $k$-space. Combining \cite[Theorems~7.1,9.4 and Lemma 9.2]{Steenrod1967} shows that $|S|^{n} \to |S|$ is a closed cofibration for all $n \geq 0$ and that $|S|$ is a Hausdorff $k$-space.
\end{proof}

We know that $\cB\cG$ is a topological group from  \cref{crossed module cat gives top group after realization}. 

Recall from \cref{lem:nerve_of_crossed_mod_and_bar_construction} that $N\cG$ is naturally isomorphic to $B_*(F, C_H, *_H)$. It is easily checked using \cite[Theorem~5.4.6]{tomDieck2008AlgebraicTopology} that $B_*(F, C_H, *_H)$ is a Reedy cofibrant simplicial space. The above lemma shows that $|B_\ast(F, C_H, *_H)|^{0} \to |B_*(F, C_H, \ast_H)|$ is a closed cofibration and that $|B_\ast(F, C_H, *_H)|$ is Hausdorff. From the homeomorphisms $G \cong B_0(F, C_H, \ast_H) \cong  |B_\ast(F, C_H, \ast_H)|^{0}$ and the well-pointedness of $G$, it follows that  $\cB\cG \cong |B_\ast(F, C_H, \ast_H)|$ is well-pointed. This proves the first part.

The second part follows by applying the above lemma to the nerve of the well-pointed Hausdorff $k$-group $\cB\cG$ (to see that this is Reedy cofibrant, use \cite[Theorem~5.4.6]{tomDieck2008AlgebraicTopology} again).

The third part follows by the Reedy cofibrancy of $N\cB\cG$ and by \cref{fat thin weak equiv for reedy cofib}.
\end{proof}

\AtNextBibliography{\small}
\printbibliography

\end{document}